\documentclass[a4paper,11pt]{amsart}

\usepackage{amsmath}
\usepackage{amssymb}
\usepackage{amsthm}
\usepackage{comment}
\usepackage{enumitem}
\usepackage{float}
\usepackage{mathtools}
\usepackage{tikz-cd}
\usepackage{xcolor}

\usepackage{tikz}
\usetikzlibrary{arrows.meta,decorations.markings}

\usepackage[textsize=footnotesize]{todonotes}

\usepackage[hidelinks]{hyperref}

\newtheoremstyle{lemnum}{\medskipamount}{\topsep}{\itshape}{}{\bfseries}{.}{ }{\thmname{#1}\thmnote{ \bfseries #3}}

\newtheorem*{thm*}{Theorem}
\newtheorem*{lem*}{Lemma}
\newtheorem{thm}{Theorem}[section]
\newtheorem{lem}[thm]{Lemma}
\newtheorem{prop}[thm]{Proposition}
\newtheorem{cor}[thm]{Corollary}
\theoremstyle{lemnum}
\newtheorem{lemnum}[thm]{Lemma}
\theoremstyle{definition}
\newtheorem{df}[thm]{Definition}
\newtheorem{ex}[thm]{Example}
\theoremstyle{remark}
\newtheorem{rmk}[thm]{Remark}

\makeatletter
\g@addto@macro\bfseries{\boldmath}
\makeatother

\numberwithin{equation}{section}

\DeclareMathOperator{\Coin}{Coin} 
\DeclareMathOperator{\coin}{coin} 
\DeclareMathOperator{\Fix}{Fix}   
\DeclareMathOperator{\id}{id}	  
\DeclareMathOperator{\ind}{ind}   
\DeclareMathOperator{\MC}{MC}	  
\DeclareMathOperator{\MF}{MF}	  
\renewcommand{\MR}{\mathrm{MR}}	  
\DeclareMathOperator{\pr}{pr}	  
\DeclareMathOperator{\Root}{Root} 

\newcommand{\QQ}{\mathbb{Q}}      
\newcommand{\RR}{\mathbb{R}}      
\newcommand{\ZZ}{\mathbb{Z}}      

\newcommand{\g}{\mathfrak{g}}     
\newcommand{\orb}{\backslash}     
\newcommand{\R}{\mathcal{R}}      
\renewcommand{\S}{\Sigma}         

\newcommand{\eps}{\varepsilon}
\renewcommand{\iff}{\enskip\Leftrightarrow\enskip}

\title[Nielsen coincidence theory for $(n,1)$-valued pairs]{Nielsen coincidence theory \\ for $(n,1)$-valued pairs}
\author{Karel Dekimpe and Lore De Weerdt}
\thanks{Research supported by Methusalem grant METH/21/03 -- long term structural funding of the Flemish Government.}
\address{KU Leuven Campus Kulak Kortrijk, 8500 Kortrijk, Belgium}
\email{Karel.Dekimpe@kuleuven.be}
\email{Lore.DeWeerdt@kuleuven.be}
\begin{document}

\maketitle

\begin{abstract}
We generalise Nielsen theory to coincidences of pairs $(f,g)$ where $f:X\multimap Y$ is $n$-valued multimap and $g:X\to Y$ is a single-valued map, for $X$ and $Y$ closed oriented triangulable manifolds of equal dimension. We prove a Wecken theorem in this setting, and formulas for the Nielsen, Lefschetz and Reidemeister numbers in terms of the analogous invariants for single-valued maps. If $X$ and $Y$ are orientable infra-nilmanifolds, we derive explicit formulas in terms of the fundamental group morphisms of $f$ and $g$.
\end{abstract}

\section{Introduction}

Let $f:X\to X$ be a self-map of a topological space $X$. In classical Nielsen fixed point theory, one is concerned with the number \[
\MF[f]\vcentcolon=\min\{\# \Fix(g) \mid g\simeq f \}
\]
where $\Fix(g)=\{x\in X\mid g(x)=x\}$ is the set of fixed points of $g$, and $\simeq$ denotes homotopy.
If $X$ is a finite polyhedron (or more generally a compact ANR), a homotopy-invariant lower bound for $\MF[f]$ is the Nielsen number $N(f)$, defined in the following two steps (for more details, see e.g.\ \cite{jiang}). First, the fixed points of $f$ are subdivided into \emph{fixed point classes}, defined using an equivalence relation on the set of lifts of $f$ to the universal covering space of $X$. To each fixed point class, one can associate an integer called the \emph{fixed point index}, which is zero if the fixed point class can become empty after a homotopy. The Nielsen number is the number of fixed point classes with non-zero index (called \emph{essential} fixed point classes). A celebrated theorem due to Wecken \cite{wecken3} is:

\begin{thm}[Wecken Theorem]
If $X$ is a compact manifold of dimension at least $3$, then $N(f)=\MF[f]$ for any map $f:X\to X$.
\end{thm}

To show this, Wecken used another definition of fixed point class, where two fixed points lie in the same class if there is a path $c$ between them so that the paths $fc$ and $c$ are homotopic\footnote{With a homotopy of paths, we will always mean a homotopy relative to the endpoints.}, which has become a more commonly used definition for generalisations. (This definition is almost equivalent to the one using lifts: the non-empty classes coincide for both definitions, but in the definition using lifts, some classes may be emtpy, and these are not detected by the definition using paths.)

In order to motivate our work, we review some of the settings to which Nielsen theory has been generalised.

\subsection{Coincidence theory}

For two maps $f,g:X\to Y$, one can consider the set of coincidence points \[
\Coin(f,g)=\{x\in X\mid f(x)=g(x)\}
\]
and study the number \[
\MC[f,g]=\min\{\# \Coin(f',g') \mid f'\simeq f,g'\simeq g \}.
\]
Fixed point classes can be generalised to coincidence classes, by calling two coincidence points equivalent if there is a path $c$ between them such that $fc\simeq gc$. To define essentiality of a coincidence class, one usually restricts to the setting where $X$ and $Y$ are closed orientable manifolds of the same dimension, where an integer-valued coincidence index is defined. In this setting, Schirmer \cite{schirmer-coin} defined the Nielsen number $N(f,g)$ as the number of coincidence classes with non-zero index, and showed that it satisfies $N(f,g)=\MC[f,g]$ in dimension at least $3$.

\subsection{Root theory}

For a map $f:X\to Y$ and a point $a\in Y$, one can consider the set of roots of $f$ at $a$, \[
\Root(f)=\{x\in X\mid f(x)=a \}.
\]
A Nielsen theory for roots, concerned with the number \[
\MR[f]=\min\{\#\Root(g) \mid g\simeq f \},
\]
has been developed by Brooks \cite{brooks-phd}. Rather than working with an index, he defines root classes (as well as coincidence classes in general) to be essential if they cannot vanish after a homotopy. This geometric notion of essentiality can be defined without the strong restrictions on the spaces $X$ and $Y$ required for an index, but the resulting Nielsen number is harder to compute.

Note that fixed points can be viewed as a special case of coincidences by taking $X=Y$ and $g:X\to X$ the identity map, and roots can be viewed as a special case of coincidences by taking $g:X\to Y$ the constant map $x\mapsto a$. But the fixed point and root problems described above are not necessarily special cases of the coincidence problem, since for $\MC[f,g]$, both $f$ and $g$ may be deformed, while in $\MF[f]$ and $\MR[f]$ the identity map resp.\ the constant map stays fixed. However, Brooks \cite{brooks1971} showed that whenever $Y$ is a topological manifold, one has \[
\MC[f,g]=\MC(g)[f]\vcentcolon=\min\{\# \Coin(f',g) \mid f'\simeq f \}.
\]
Thus, in that setting, fixed point theory and root theory can really be viewed as special cases of coincidence theory.

\subsection{Fixed point theory of $n$-valued maps}

Another direction in which Nielsen theory has been generalised, again by Schirmer \cite{schirmer2}, is that of fixed points of $n$-valued self-maps $f:X\multimap X$ of finite polyhedra. An $n$-valued map is an upper- and lower-semicontinuous multifunction where $f(x)\subseteq X$ has cardinality exactly $n$ for every $x\in X$ (see e.g.\ \cite{browngoncalves}). In this setting, one defines \[
\Fix(f)=\{x\in X\mid x\in f(x)\},
\] 
and the number $\MF[f]$ accordingly. A property of $n$-valued maps which the theory in \cite{schirmer2} heavily relies on, is that when $X$ is simply connected, any $n$-valued map $f:X\multimap Y$ \emph{splits} into $n$ single-valued maps, written $f=\{f_1,\ldots,f_n\}$. This holds in particular when $X$ is the unit interval $I$, so for any path $c:I\to X$, the image $fc$ splits into $n$ paths in $Y$. Fixed points of an $n$-valued map $f:X\multimap X$ are then called equivalent if there exists a path $c$ between them such that, if $fc=\{c_1,\ldots,c_n\}$, one of the paths $c_i$ is homotopic to $c$. To define an index, Schirmer approximates the map $f$ by a map $f'$ with finitely many fixed points. Around each fixed point $x$, she can then take an isolating neighborhood on which $f'$ splits, so that $x$ is a fixed point of precisely one of the splitting factors, and the index of $x$ can be defined as the single-valued index of this factor at its isolated fixed point $x$. The resulting Nielsen number $N(f)$ is a well-defined homotopy-invariant lower bound for $\MF[f]$, and is equal to $\MF[f]$ in case $X$ is a compact manifold of dimension at least $3$ (see \cite{schirmer3}).

\subsection{Coincidence and root theory of $n$-valued maps}

Generalisations to coincidences and roots of $n$-valued maps have been defined by Brown and Kolahi \cite{brownkolahi}, but only using Brooks' geometric definition of essentiality, and without showing any results regarding the sharpness of the bound $N(f,g)\leq \MC[f,g]$. Moreover, there is a mistake in their paper which makes the Nielsen number (for coincidences in general) ill-defined. 

For an $n$-valued map $f:X\multimap Y$ and an $m$-valued map $g:X\multimap Y$, they define \[
\Coin(f,g)=\{x\in X \mid f(x)\cap g(x)\neq \varnothing \}.
\]
Two points in $\Coin(f,g)$ are called equivalent if there is a path $c:I\to X$ such that the multimaps $fc:I\multimap Y$ and $gc:I\multimap Y$ split as $\{c_1,\ldots,c_n \}$ and $\{d_1,\ldots,d_m \}$ respectively, where for some $i\in \{1,\ldots,n\}$ and $j\in \{1,\ldots,m\}$, the paths $c_i$ and $d_j$ are homotopic. 
However, this is in general not an equivalence relation, as the following example demonstrates.

\begin{ex}
Let $f,g:S^1\multimap S^1$ be the $2$-valued maps of the circle given by the graphs below (where $S^1$ is represented as an interval with the endpoints identified).

\begin{figure}[h!]
\begin{tikzpicture}[scale=3]
\draw (0,0) rectangle (1,1);
\draw[thick,red] (0,0.33)--(1,0.33);
\draw[thick,red] (0,0.67)--(1,0.67);
\draw[thick,blue] (0,0)--(0.33,0.33)--(0.5,0.33)--(1,0.5);
\draw[thick,blue] (0,0.5)--(0.5,0.67)--(0.67,0.67)--(1,1);
\draw[red] (0.15,0.75) node{\small$f$};
\draw[blue] (0.25,0.5) node{\small$g$};
\draw[dashed] (0.33,0)node[below]{\small$x_0$}--(0.33,0.33);
\draw[dashed] (0.5,0)node[below]{\small$x_1$}--(0.5,0.67);
\draw[dashed] (0.67,0)node[below]{\small$x_2$}--(0.67,0.67);
\end{tikzpicture}
\end{figure}

Clearly, the coincidence point $x_1$ is equivalent both to $x_0$ and to $x_2$, since in both cases, if $c$ is the line segment between them, there is a factor of $fc$ and a factor of $gc$ that have the same graph and are clearly homotopic. On the other hand, $x_0$ is not equivalent to $x_2$, since for any path $c$ between them, there is no factor of $fc$ with the same begin and endpoint as a factor of $gc$.

Thus, the coincidence relation is not transitive. In terms of coincidence classes, there would be two classes, and $x_1$ would belong to both of them.
\end{ex}

Since a point can belong to multiple coincidence classes, also the statement of \cite[Lemma 2.1]{brownkolahi} that, for a homotopy $(F,G)=\{(f_t,g_t)\mid t\in I\}$, each coincidence class of $(f_0,g_0)$ lies in a \textit{unique} coincidence class of $(F,G)$ is false, rendering their definition of essentiality of a coincidence class (for every homotopy, its unique coincidence class containing this class has non-empty intersection with $X\times \{1\}$) and therefore the Nielsen number (the number of essential coincidence classes) ill-defined.

Note that this problem cannot occur as soon as one of the maps, say $g$, is single-valued: for $x_0,x_1,x_2\in \Coin(f,g)$, suppose $c$ is a path from $x_0$ to $x_1$ such that $fc=\{c_1,\ldots,c_n\}$ with $c_i\simeq gc$ for some $i$, and $d$ is a path from $x_1$ to $x_2$ such that $fd=\{d_1,\ldots,d_n\}$ with $d_j\simeq gd$ for some $j$. The concatenation $e$ of $c$ and $d$ is a path from $x_0$ to $x_2$; write $fe=\{e_1,\ldots,e_n\}$. Since $c_i(1)=gc(1)=gd(0)=d_j(0)$, one of the factors $e_k$ must be the concatenation of $c_i$ and $d_j$. Since $c_i\simeq gc$ and $d_j\simeq gd$, the concatenation $e_k$ of $c_i$ and $d_j$ is homotopic to the concatenation of $gc$ and $gd$, which is $ge$.

The purpose of this paper is to develop a Nielsen theory in this setting, i.e.\ for coincidences of a pair $(f,g)$ where $f:X\multimap Y$ is an $n$-valued map and $g:X\to Y$ is a single-valued map.

\subsection{The coincidence problem for $(n,1)$-valued pairs}

For an $n$-valued map $f:X\multimap Y$ and a single-valued map $g:X\to Y$, we consider \[
\Coin(f,g)=\{x\in X \mid g(x)\in f(x) \}.
\]
Special cases of this are coincidences (and therefore also fixed points and roots) of single-valued maps (by taking $n=1$), fixed points of $n$-valued maps (for $X=Y$ and $g:X\to X$ the identity map), and roots of $n$-valued maps (i.e., for a constant $a\in Y$, the points $x\in X$ such that $a\in f(x)$, so the coincidences of $f$ with the constant single-valued map $x\mapsto a$). Considering the problem pointed out in the previous section, this is the most general setting containing all of the above ones as special cases, in which a Nielsen theory can be defined.

First of all, we observe the following useful fact about the coincidence set defined above:

\begin{lem}\label{lem:coin-closed}
If $Y$ is Hausdorff, then $\Coin(f,g)$ is closed.
\end{lem}

\begin{proof}
It is easy to see that, given an $n$-valued map $f:X\to Y$ and an $m$-valued map $g:X\to Z$, the multifunction \[
(f,g):X\to Y\times Z:x\mapsto \{(y,z) \mid y\in f(x), z\in g(x) \}
\]
is $nm$-valued and continuous. In particular, in our setting, we can consider the $n$-valued map \[
(f,g):X\multimap Y\times Y:x\mapsto \{(y,g(x)) \mid y\in f(x) \}.
\]
Since $Y$ is Hausdorff, the diagonal $\Delta=\{(y,y)\mid y\in Y\}$ is closed in $Y\times Y$. Since $(f,g)$ is continuous (in particular, upper-semicontinuous), the (lower) inverse image \[
(f,g)^{-1}(\Delta)=\{x\in X\mid (f,g)(x)\cap \Delta\neq \varnothing\}=\Coin(f,g)
\]
is closed as well.
\end{proof}

We can again consider \[
\MC[f,g]=\min\{\# \Coin(f',g') \mid f'\simeq f,g'\simeq g \}.
\]
As in the single-valued case, we can show that if $Y$ is a manifold, \[
\MC[f,g]=\MC(g)[f]\vcentcolon=\min\{\#\Coin(f',g)\mid f'\simeq f\},
\]
so that the $n$-valued fixed point and root problems are special cases of the coincidence problem.
This follows from the theorem below: indeed, given an $n$-valued map $f:X\multimap Y$ and a map $g:X\to Y$, for any two homotopies $f'\simeq f$ and $g'\simeq g$ there exists an $f''\simeq f$ such that $\Coin(f',g')=\Coin(f'',g)$ -- this follows by applying the following theorem to $f_0=f'$, $g_0=g'$ and $g_1=g$.

\begin{thm}\label{thm:MC}
Let $X$ be a topological space and $Y$ a topological manifold. Let $f_0:X\multimap Y$ be an $n$-valued map and $g_0:X\to Y$ a map. For any homo\-topy $\{g_t:X\to Y\mid t\in I\}$, there exists a homotopy $\{f_t:X\multimap Y\mid t\in I\}$ of $n$-valued maps such that \[
\Coin(f_t,g_t)=\Coin(f_0,g_0)
\]
for all $t\in I$.
\end{thm}

\begin{proof}
In \cite[Proof of Lemma, p.~49]{brooks1971}, Brooks shows that if $Y$ is a topological manifold, then for any homotopy $\{g_t:X\to Y\mid t\in I\}$, there exists a homeomorphism \[
H:X\times I\times Y\to X\times I\times Y
\] 
such that $H(x,t,g_0(x))=(x,t,g_t(x))$ for all $(x,t)\in X\times I$, and $p\circ H=p$, where $p:X\times I\times Y\to X\times I$ is the projection. It follows that there are maps $h:X\times I\times Y\to Y$ and $k:X\times I\times Y\to Y$ such that \begin{align*}
H(x,t,y)&=(x,t,h(x,t,y)) \\ H^{-1}(x,t,y)&=(x,t,k(x,t,y))
\end{align*}
for $(x,t,y)\in X\times I\times Y$, with $h(x,t,g_0(x))=g_t(x)$ and $k(x,t,g_t(x))=g_0(x)$ for all $(x,t)\in X\times I$, and \[
h(x,t,k(x,t,y))=y=k(x,t,h(x,t,y))
\]
for all $(x,t,y)\in X\times I\times Y$, since \begin{align*}
(x,t,h(x,t,k(x,t,y)))=H(x,t&,k(x,t,y))=H(H^{-1}(x,t,y))=(x,t,y) \\
(x,t,k(x,t,h(x,t,y)))=H^{-1}(x&,t,h(x,t,y))=H^{-1}(H(x,t,y))=(x,t,y).
\end{align*}

We can then define \[
f_t:X\multimap Y:x\mapsto \{h(x,t,k(x,0,y)) \mid y\in f_0(x) \}
\]
for all $t\in I$. This function is continuous, since $f_0$, $h$ and $k$ are continuous. It is $n$-valued, for if $h(x,t,k(x,0,y))=h(x,t,k(x,0,y'))$, then \begin{align*}
y&=h(x,0,k(x,t,h(x,t,k(x,0,y)))) \\
&=h(x,0,k(x,t,h(x,t,k(x,0,y'))))=y'.
\end{align*}
For $t=0$, this is the original map $f_0$: for all $x\in X$, \[
\{h(x,0,k(x,0,y))\mid y\in f_0(x)\}=\{y\mid y\in f_0(x)\}=f_0(x).
\]

Fix $t\in I$. To show $\Coin(f_0,g_0)\subseteq \Coin(f_t,g_t)$, suppose $g_0(x)\in f_0(x)$. Then \begin{align*}
g_t(x)&=h(x,t,g_0(x))\\&=h(x,t,k(x,0,g_0(x)))\\&\in \{h(x,t,k(x,0,y)) \mid y\in f_0(x) \}=f_t(x).
\end{align*}
Conversely, suppose $g_t(x)\in f_t(x)$. Then $g_t(x)=h(x,t,k(x,0,y))$ for some $y\in f_0(x)$, so \begin{align*}
g_0(x)&=h(x,0,g_0(x))\\&=h(x,0,k(x,t,g_t(x)))\\&=h(x,0,k(x,t,h(x,t,k(x,0,y))))\\&=h(x,0,k(x,0,y))\\&=y\in f_0(x). \qedhere
\end{align*}
\end{proof}

\subsection{Overview of the paper}

This paper is organised as follows. In section \ref{sec:def}, we define coincidence classes, a coincidence index and the Nielsen number $N(f,g)$. Rather than defining coincidence classes using paths as described above, we define them using lifts as in classical Nielsen theory, and then observe that this is indeed equivalent to the definition using paths (for the non-empty classes). Along the way, we also establish a link with a generalised notion of Reidemeister classes, similar to \cite{charlotte1}.

In order to define an index, we restrict to closed orientable manifolds of equal dimension. We do not use the same approach as Schirmer, but instead we define the index in terms of the single-valued indices of lifts of $f$ and $g$ to a finite covering space for which the lift of $f$ splits. This idea was already exploited by Brown and Gon\c{c}alves \cite{browngoncalves2023} in the setting of roots (although that paper also contains a mistake, which we address in section \ref{sec:roots}).

In section \ref{sec:axioms}, we show two classical results. The first one is the Wecken property: if $X$ and $Y$ are manifolds of dimension at least $3$, the Nielsen number satisfies $N(f,g)=\MC[f,g]$. The second result is that the coincidence index is uniquely determined by three natural axioms, as in the fixed point case \cite{staecker2007}. From our proof follows an equivalent definition of the index which is the natural generalisation of Schirmer's definition. Throughout this third section, we use simplicial approximation results due to Schirmer, for which we must restrict to triangulable manifolds.

Sections \ref{sec:av} and \ref{sec:inm} are concerned with the computation of the Nielsen number. In section \ref{sec:av}, we show that the Nielsen number can be expressed in terms of those of the single-valued lifts to the finite covering space mentioned above. In the setting of infra-nilmanifolds, where explicit formulas for single-valued Nielsen coincidence numbers exist, we can use this to give explict formulas for the Nielsen numbers of all $(n,1)$-valued pairs; this is the content of section \ref{sec:inm}. 

Lastly, in section \ref{sec:roots}, we consider the special case of roots. A theory of Nielsen and Reidemeister classes for roots of $n$-valued maps is given in the paper \cite{browngoncalves2023} mentioned above. We point out and correct a mistake in that paper, and show how the resulting theory can be viewed as a special case of ours.

\vspace{-\bigskipamount}

\section{Definitions}\label{sec:def}

Let $X$ and $Y$ be closed orientable triangulable manifolds of equal dimension. Let $f:X\multimap Y$ be an $n$-valued map and $g:X\to Y$ a single-valued map.

In order to define coincidence classes using lifts, let $p:\tilde{X}\to X$ and $q:\tilde{Y}\to Y$ be the universal covering spaces of $X$ and $Y$, with respective covering groups $\Pi$ and $\Delta$. 
Any single-valued map $g:X\to Y$ admits a lift $\tilde{g}:\tilde{X}\to \tilde{Y}$ such that \[
\begin{tikzcd}
\tilde{X} \ar[r,"\tilde{g}"] \ar[d,"p"'] & \tilde{Y} \ar[d,"q"] \\
X \ar[r,"g"] & Y
\end{tikzcd}
\]
commutes. Each such lift $\tilde{g}$ induces a morphism $\psi:\Pi\to \Delta$ of the covering groups defined by \[
\forall \gamma\in \Pi: \tilde{g}\gamma=\psi(\gamma)\tilde{g}.
\]

Lifts and covering group morphisms for $n$-valued maps can be defined as in \cite{charlotte1}, using the equivalent and more practical definition for $n$-valued maps given in \cite{browngoncalves}: an $n$-valued multifunction $f:X\multimap Y$ is continuous if and only if the corresponding single-valued function \[
X\to D_n(Y):x\mapsto f(x)
\]
is continuous, where \[
D_n(Y)=\{\{y_1,\ldots,y_n\} \subseteq Y\mid y_i\neq y_j \text{ if } i\neq j \}
\]
is the unordered configuration space, topologised as the quotient of the ordered configuration space \[
F_n(Y)=\{(y_1,\ldots,y_n)\in Y^n \mid y_i\neq y_j \text{ if } i\neq j \}
\]
under the action of the symmetric group $\Sigma_n$. Thus, an $n$-valued map $f:X\multimap Y$ can be viewed as a continuous function $X\to D_n(Y)$, which is denoted by $f$ as well.

A covering space for $D_n(Y)$ that is very useful to study $n$-valued maps is the \emph{orbit configuration space} \[
F_n(\tilde{Y},\Delta)=\{(\tilde{y}_1,\ldots,\tilde{y}_n)\in \tilde{Y}^n \mid \Delta\tilde{y}_i\neq \Delta\tilde{y}_j \text{ if } i\neq j \}
\]
with covering map \[
q^n:F_n(\tilde{Y},\Delta)\to D_n(Y):(\tilde{y}_1,\ldots,\tilde{y}_n)\mapsto \{q(\tilde{y}_1),\ldots,q(\tilde{y}_n) \}.
\]
The group of covering transformations is $\Delta^n\rtimes \Sigma_n$, where the group operation is given by \[
(\alpha_1,\ldots,\alpha_n;\sigma)(\beta_1,\ldots,\beta_n;\tau)=(\alpha_1\beta_{\sigma^{-1}(1)},\ldots,\alpha_n\beta_{\sigma^{-1}(n)};\sigma\tau)
\]
and where $(\alpha_1,\ldots,\alpha_n;\sigma)\in\Delta^n\rtimes \Sigma_n$ acts on $(\tilde{y}_1,\ldots,\tilde{y}_n)\in F_n(\tilde{Y},\Delta)$ by \[
(\alpha_1,\ldots,\alpha_n;\sigma)(\tilde{y}_1,\ldots,\tilde{y}_n)=(\alpha_1\tilde{y}_{\sigma^{-1}(1)},\ldots,\alpha_n\tilde{y}_{\sigma^{-1}(n)}).
\]

Any $n$-valued map $f:X\to D_n(Y)$ admits a lift $\tilde{f}:\tilde{X}\to F_n(\tilde{Y},\Delta)$ such that \[
\begin{tikzcd}
\tilde{X} \ar[r,"\tilde{f}"] \ar[d,"p"'] & F_n(\tilde{Y},\Delta) \ar[d,"q^n"] \\
X \ar[r,"f"] & D_n(Y)
\end{tikzcd}
\]
commutes. Since $F_n(\tilde{Y},\Delta)$ is a subspace of $\tilde{Y}^n$, such a lift decomposes as $\tilde{f}=(\tilde{f}_1,\ldots,\tilde{f}_n)$ for single-valued maps $\tilde{f}_i:\tilde{X}\to\tilde{Y}$, called the \emph{lift factors} of $f$. Every lift $\tilde{f}$ induces a morphism $\varphi:\Pi\to \Delta^n\rtimes \Sigma_n$ of the covering groups defined by \[
\forall \gamma\in \Pi:\tilde{f}\gamma=\varphi(\gamma)\tilde{f}.
\]
If we write $\varphi=(\varphi_1,\ldots,\varphi_n;\sigma)$ with $\varphi_i:\Pi\to \Delta$ and $\sigma:\Pi\to \Sigma_n$, this is equivalent to \[
(\tilde{f}_1\gamma,\ldots,\tilde{f}_n\gamma)=(\varphi_1(\gamma)\tilde{f}_{\sigma_\gamma^{-1}(1)},\ldots,\varphi_n(\gamma)\tilde{f}_{\sigma_\gamma^{-1}(n)})
\]
where we write $\sigma_\gamma=\sigma(\gamma)\in \Sigma_n$. The map $\sigma:\Pi\to \Sigma_n$ is a group morphism; the maps $\varphi_i:\Pi\to \Delta$ in general are not, but their restrictions to the subgroups $S_i=\{\gamma\in \Pi\mid \sigma_\gamma(i)=i \}$ are.

\subsection{Coincidence classes}\label{subsec:cc}

Fix maps $f:X\to D_n(Y)$ and $g:X\to Y$.

\begin{df}
Let $\tilde{f}=(\tilde{f}_1,\ldots,\tilde{f}_n)$ be a lift of $f$ and $\tilde{g}$ a lift of $g$: 
\[
\begin{tikzcd}
\tilde{X} \ar[r,"\tilde{f}"] \ar[d,"p"'] & F_n(\tilde{Y},\Delta) \ar[d,"q^n"] & \tilde{X} \ar[r,"\tilde{g}"] \ar[d,"p"'] & \tilde{Y} \ar[d,"q"] \\
X \ar[r,"f"] & D_n(Y) & X \ar[r,"g"] & Y.
\end{tikzcd}
\]
For all $i$, we call $(\tilde{f}_i,\tilde{g})$ a \emph{lifting pair} of $(f,g)$.
\end{df}

Note that \[
\Coin(f,g)=\bigcup_{(\tilde{f}_i,\tilde{g})} p\Coin(\tilde{f}_i,\tilde{g})
\]
where the union runs over all lifting pairs of $(f,g)$.

\begin{lem}\label{lem:unique-lift}
If $(\tilde{f}_i,\tilde{g})$ and $(\tilde{f}'_j,\tilde{g}')$ are two lifting pairs of $(f,g)$ such that $\tilde{f}_i(\tilde{x})=\tilde{f}'_j(\tilde{x})$ and $\tilde{g}(\tilde{x})=\tilde{g}'(\tilde{x})$ for some $\tilde{x}\in \tilde{X}$, then $(\tilde{f}_i,\tilde{g})=(\tilde{f}'_j,\tilde{g}')$.
\end{lem}

\begin{proof}
It follows from standard covering space theory that $\tilde{g}=\tilde{g}'$. We also know $\tilde{f}'=\alpha\tilde{f}$ for some $\alpha\in \Delta^n\rtimes \S_n$; in particular, $\tilde{f}'_j=\delta\tilde{f}_k$ for some $\delta\in \Delta$ and $k\in \{1,\ldots,n\}$. Then we have \[
q(\tilde{f}_k(\tilde{x}))=q(\delta\tilde{f}_k(\tilde{x}))=q(\tilde{f}'_j(\tilde{x}))=q(\tilde{f}_i(\tilde{x})),
\]
from which it follows that $k=i$. But then $\delta$ is the unique covering translation such that $\delta\tilde{f}_i(\tilde{x})=\tilde{f}_i(\tilde{x})$, which means $\delta=1$.
\end{proof}

\begin{lem}
For all $\gamma\in \Pi$ and $\delta\in \Delta$, if $(\tilde{f}_i,\tilde{g})$ is a lifting pair of $(f,g)$, then so is $(\delta\tilde{f}_i\gamma^{-1},\delta\tilde{g}\gamma^{-1})$.
\end{lem}

\begin{proof}
It follows from standard covering space theory that $\delta\tilde{g}\gamma^{-1}$ is a lift of $g$ and $(\delta,\ldots,\delta;1)\tilde{f}\gamma^{-1}$ is a lift of $f$; $\delta\tilde{f}_i\gamma^{-1}$ is a lift-factor of the latter.
\end{proof}

\begin{df}
Define an equivalence relation on the set of lifting pairs: \[
(\tilde{f}_i,\tilde{g}) \sim (\tilde{f}'_j,\tilde{g}') \iff \exists \gamma\in \Pi,\delta\in \Delta: (\tilde{f}'_j,\tilde{g}')=(\delta\tilde{f}_i\gamma^{-1},\delta\tilde{g}\gamma^{-1}).
\]
\end{df}

\begin{prop}\label{prop:cc-partition}
Let $(\tilde{f}_i,\tilde{g})$ and $(\tilde{f}'_j,\tilde{g}')$ be lifting pairs of $(f,g)$.
\begin{itemize}
\item[\rm(i)] If $(\tilde{f}_i,\tilde{g})\sim (\tilde{f}'_j,\tilde{g}')$, then $p\Coin(\tilde{f}_i,\tilde{g})=p\Coin(\tilde{f}'_j,\tilde{g}')$;
\item[\rm(ii)] If $(\tilde{f}_i,\tilde{g})\not\sim (\tilde{f}'_j,\tilde{g}')$, then $p\Coin(\tilde{f}_i,\tilde{g})\cap p\Coin(\tilde{f}'_j,\tilde{g}')=\varnothing$.
\end{itemize}
\end{prop}

\begin{proof}
For (i), suppose $(\tilde{f}'_j,\tilde{g}')=(\delta\tilde{f}_i\gamma^{-1},\delta\tilde{g}\gamma^{-1})$, and take $x\in p\Coin(\tilde{f}_i,\tilde{g})$; we will show that $x\in p\Coin(\tilde{f}'_j,\tilde{g}')$. The other inclusion follows by symmetry.

Take $\tilde{x}\in \Coin(\tilde{f}_i,\tilde{g})$ with $p(\tilde{x})=x$. Then \[
\tilde{f}'_j(\gamma\tilde{x})=\delta\tilde{f}_i(\tilde{x})=\delta\tilde{g}(\tilde{x})=\tilde{g}'(\gamma\tilde{x}).
\]
So $\gamma\tilde{x}\in \Coin(\tilde{f}'_j,\tilde{g}')$, and $p(\gamma\tilde{x})=p(\tilde{x})=x$, which means $x\in p\Coin(\tilde{f}'_j,\tilde{g}')$.

For (ii), suppose $x\in p\Coin(\tilde{f}_i,\tilde{g})\cap p\Coin(\tilde{f}'_j,\tilde{g}')$. Take $\tilde{x}\in \Coin(\tilde{f}_i,\tilde{g})$ and $\tilde{x}'\in \Coin(\tilde{f}'_j,\tilde{g}')$ with $p(\tilde{x})=p(\tilde{x}')=x$. Then \[
q(\tilde{g}(\tilde{x})) = g(p(\tilde{x})) = g(p(\tilde{x}')) = q(\tilde{g}'(\tilde{x}'))
\]
so $\tilde{g}'(\tilde{x}')=\delta\tilde{g}(\tilde{x})$ for some $\delta\in \Delta$. Since $\tilde{f}_i(\tilde{x})=\tilde{g}(\tilde{x})$ and $\tilde{f}'_j(\tilde{x}')=\tilde{g}'(\tilde{x}')$, also $\tilde{f}'_j(\tilde{x}')=\delta \tilde{f}_i(\tilde{x})$. Since $p(\tilde{x})=p(\tilde{x}')$, there is a $\gamma\in \Pi$ such that $\tilde{x}'=\gamma\tilde{x}$, so we get $\tilde{g}'(\tilde{x}')=\delta\tilde{g}\gamma^{-1}(\tilde{x}')$ and $\tilde{f}'_j(\tilde{x}')=\delta \tilde{f}_i\gamma^{-1}(\tilde{x}')$. By Lemma \ref{lem:unique-lift} it follows that $(\tilde{f}'_j,\tilde{g}')=(\delta\tilde{f}_i\gamma^{-1},\delta\tilde{g}\gamma^{-1})$.
\end{proof}

Thus, we have a partition (where some of the sets may be empty) \[
\Coin(f,g)= \bigsqcup_{[(\tilde{f}_i,\tilde{g})]} p\Coin(\tilde{f}_i,\tilde{g})
\]
where $[(\tilde{f}_i,\tilde{g})]$ denote the equivalence classes for the relation $\sim$.

\begin{df}
We call the classes $[(\tilde{f}_i,\tilde{g})]$ the \emph{lifting classes}, and the corresponding sets $p\Coin(\tilde{f}_i,\tilde{g})$ the \emph{coincidence classes} of the pair $(f,g)$.
\end{df}

Now fix a lift $\tilde{f}:\tilde{X}\to F_n(\tilde{Y},\Delta)$ for $f$ and a lift $\tilde{g}:\tilde{X}\to \tilde{Y}$ for $g$.

\begin{prop}
For any lifting pair $(\tilde{f}'_j,\tilde{g}')$, there exist $i\in \{1,\ldots,n\}$ and $\alpha\in \Delta$ such that $(\tilde{f}'_j,\tilde{g}')\sim (\alpha\tilde{f}_i,\tilde{g})$.
\end{prop}

\begin{proof}
As in the proof of Lemma \ref{lem:unique-lift}, we can write $\tilde{f}'_j=\delta\tilde{f}_i$ for some $\delta\in \Delta$ and $i\in \{1,\ldots,n\}$. By standard covering space theory, we can write $\tilde{g}'=\delta'\tilde{g}$ for some $\delta'\in \Delta$. For $\alpha=(\delta')^{-1}\delta$, we have $(\tilde{f}'_j,\tilde{g}')\sim (\alpha\tilde{f}_i,\tilde{g})$.
\end{proof}

Let $\varphi=(\varphi_1,\ldots,\varphi_n;\sigma):\Pi\to \Delta^n\rtimes \S_n$ and $\psi:\Pi\to \Delta$ be the covering group morphisms induced by $\tilde{f}$ and $\tilde{g}$, so that for all $\gamma\in \Pi$ \[
\tilde{f}\gamma=\varphi(\gamma)\tilde{f}=(\varphi_1(\gamma)\tilde{f}_{\sigma_\gamma^{-1}(1)},\ldots,\varphi_n(\gamma)\tilde{f}_{\sigma_\gamma^{-1}(n)}) \quad \text{and} \quad \tilde{g}\gamma=\psi(\gamma)\tilde{g}.
\]
These morphisms define an equivalence relation on $\Delta\times \{1,\ldots,n\}$ by \[
(\alpha,i)\sim_{\varphi,\psi} (\beta,j) \iff \exists \gamma\in \Pi:i=\sigma_\gamma(j),\; \alpha=\psi(\gamma)\beta\varphi_j(\gamma^{-1}).
\]
The set of equivalence classes $\Delta\times \{1,\ldots,n\}/{\sim}_{\varphi,\psi}$ will be denoted by $\R[\varphi,\psi]$.

\begin{prop}\label{prop:cc-equiv-R}
For $\alpha,\beta\in \Delta$ and $i,j\in \{1,\ldots,n\}$, we have $(\alpha\tilde{f}_i,\tilde{g}) \sim (\beta\tilde{f}_j,\tilde{g})$ if and only if $(\alpha,i)\sim_{\varphi,\psi} (\beta,j)$.
\end{prop}

\begin{proof}
Suppose $(\alpha\tilde{f}_i,\tilde{g})=(\delta\beta\tilde{f}_j\gamma^{-1},\delta\tilde{g}\gamma^{-1})$. From the equality $\tilde{g}=\delta\tilde{g}\gamma^{-1}$ it follows that $\delta\tilde{g}=\tilde{g}\gamma=\psi(\gamma)\tilde{g}$, and hence $\delta=\psi(\gamma)$. Then the equality for the lift-factors of $f$ becomes \[
\alpha\tilde{f}_i=\psi(\gamma)\beta\tilde{f}_j\gamma^{-1}=\psi(\gamma)\beta\varphi_j(\gamma^{-1})\tilde{f}_{\sigma_{\gamma^{-1}}^{-1}(j)}
\]
from which it follows that $\alpha=\psi(\gamma)\beta\varphi_j(\gamma^{-1})$ and $i=\sigma_{\gamma^{-1}}^{-1}(j)=\sigma_\gamma(j)$.

Conversely, if $i=\sigma_\gamma(j)$ and $\alpha=\psi(\gamma)\beta\varphi_j(\gamma^{-1})$, by definition of $\varphi$ and $\psi$ we have $(\alpha\tilde{f}_i,\tilde{g})=(\psi(\gamma)\beta\tilde{f}_j\gamma^{-1},\psi(\gamma)\tilde{g}\gamma^{-1})$.
\end{proof}

Thus, with the fixed lifts $\tilde{f}$ and $\tilde{g}$, \[
\Coin(f,g)= \bigsqcup_{[(\alpha,i)]\in \R[\varphi,\psi]} p\Coin(\alpha\tilde{f}_i,\tilde{g}).
\]

\begin{rmk}\label{rmk:lift-choice}
The morphisms $\varphi$ and $\psi$ depend on the choice of lifts $\tilde{f}$ and $\tilde{g}$. Concretely, for any other lift $\tilde{g}'$ of $\tilde{g}$, there is a $\delta\in \Delta$ such that $\tilde{g}'=\delta\tilde{g}$, and for any other lift $\tilde{f}'$ of $f$, there is a $(\delta_1,\ldots,\delta_n;\eta)\in \Delta^n\rtimes \S_n$ such that \[
\tilde{f}'=(\delta_1,\ldots,\delta_n;\eta)\tilde{f}=(\delta_1\tilde{f}_{\eta^{-1}(1)},\ldots,\delta_n\tilde{f}_{\eta^{-1}(n)}).
\]
The morphism $\psi'$ induced by $\tilde{g}'$ is $\tau_\delta\psi$, where $\tau_\delta:\Delta\to\Delta:g\mapsto \delta g\delta^{-1}$ is conjugation with $\delta$; the morphism $\varphi'$ induced by $\tilde{f}'$ is given by \[
\varphi'(\gamma)=(\delta_1\varphi_{\eta^{-1}(1)}(\gamma)\delta^{-1}_{\eta\sigma_\gamma^{-1}\eta^{-1}(1)},\ldots,\delta_n\varphi_{\eta^{-1}(n)}(\gamma)\delta^{-1}_{\eta\sigma_\gamma^{-1}\eta^{-1}(n)};\eta\sigma_\gamma\eta^{-1}).
\]
The corresponding equivalence relation is given by \begin{align*}
(\alpha&,i) \sim_{\varphi',\psi'} (\beta,j) \\ &\iff \exists \gamma\in \Pi: i=\eta\sigma_\gamma\eta^{-1}(j),\; \alpha=\delta\psi(\gamma)\delta^{-1}\beta\delta_j\varphi_{\eta^{-1}(j)}(\gamma^{-1})\delta^{-1}_{\eta\sigma_\gamma\eta^{-1}(j)} \\
&\iff \exists \gamma\in \Pi: \eta^{-1}(i)=\sigma_\gamma \eta^{-1}(j),\; \delta^{-1}\alpha\delta_i=\psi(\gamma)\delta^{-1}\beta\delta_j \varphi_{\eta^{-1}(j)}(\gamma^{-1}) \\
&\iff (\delta^{-1}\alpha\delta_i,\eta^{-1}(i)) \sim_{\varphi,\psi} (\delta^{-1}\beta\delta_j,\eta^{-1}(j))
\end{align*}
(where we used $\eta\sigma_\gamma\eta^{-1}(j)=i$ in the second step). The lifting pair representing the class corresponding to $[(\alpha,i)]\in \R[\varphi',\psi']$ is \[
(\alpha\tilde{f}'_i,\tilde{g}')=(\alpha \delta_i\tilde{f}_{\eta^{-1}(i)},\delta\tilde{g}),
\]
which is equivalent to the pair $(\delta^{-1}\alpha \delta_i\tilde{f}_{\eta^{-1}(i)},\tilde{g})$, so the partition of $\Coin(f,g)$ into coincidence classes corresponding to this choice of lifts is also \begin{align*}
\bigsqcup_{[(\alpha,i)]\in \R[\varphi',\psi']} p\Coin(\alpha\tilde{f}'_i,\tilde{g}') &= \bigsqcup_{[(\delta^{-1}\alpha\delta_i,\eta^{-1}(i))]\in \R[\varphi,\psi]} p\Coin(\delta^{-1}\alpha \delta_i\tilde{f}_{\eta^{-1}(i)},\tilde{g}) \\ &= \bigsqcup_{[(\alpha,i)]\in \R[\varphi,\psi]} p\Coin(\alpha \tilde{f}_i,\tilde{g}).
\end{align*}
\end{rmk}

\begin{rmk}[Coincidence classes of homotopies]\label{rmk:cc-homotopy}
If $F:X\times I\to D_n(Y)$ and $G:X\times I\to Y$ are homotopies, then for $t\in I$ we can consider the maps $f_t:X\to D_n(Y)$ and $g_t:X\to Y$ given by $f_t(x)=F(x,t)$ and $g_t(x)=G(x,t)$. We will write $(F,G)=\{(f_t,g_t)\mid t\in I\}$.
	
Note that the universal cover of $X\times I$ is $p\times \id:\tilde{X}\times I\to X\times I$, so a lifting pair of $(F,G)$ consists of maps $\tilde{F}_i$ and $\tilde{G}$ fitting into the following diagrams: \[
\begin{tikzcd}
\tilde{X}\times I \ar[r,"{(\tilde{F}_1,\ldots,\tilde{F}_n)}"] \ar[d,"p\times \id"'] &[2.3em] F_n(\tilde{Y},\Delta) \ar[d,"q^n"] & \tilde{X}\times I \ar[r,"\tilde{G}"] \ar[d,"p\times \id"'] & \tilde{Y} \ar[d,"q"] \\
X\times I \ar[r,"F"] & D_n(Y) & X\times I \ar[r,"G"] & Y.
\end{tikzcd}
\]
For all $t$, the map $((\tilde{f}_t)_1,\ldots,(\tilde{f}_t)_n)$ defined by $(\tilde{f}_t)_i(\tilde{x})=\tilde{F}_i(\tilde{x},t)$ is a lift of $f_t$, and the map $\tilde{g}_t$ defined by $\tilde{g}_t(\tilde{x})=\tilde{G}(\tilde{x},t)$ is a lift of $g_t$.

The covering group of $p\times \id$ is also $\Pi$, with action on $\tilde{X}\times I$ given by $\gamma(\tilde{x},t)=(\gamma\tilde{x},t)$. If $\varphi=(\varphi_1,\ldots,\varphi_n;\sigma):\Pi\to \Delta^n \rtimes \S_n$ is the morphism of covering groups induced by the lift $(\tilde{F}_1,\ldots,\tilde{F}_n)$ of $F$, then \[
(\tilde{f}_t)_i(\gamma\tilde{x})=\tilde{F}_i(\gamma\tilde{x},t)=\varphi_i(\gamma)\tilde{F}_{\sigma_\gamma^{-1}(i)}(\tilde{x},t)=\varphi_i(\gamma)(\tilde{f}_t)_{\sigma_\gamma^{-1}(i)}(\tilde{x})
\]
so $\varphi$ is also the morphism induced by the lift $((\tilde{f}_t)_1,\ldots,(\tilde{f}_t)_n)$ of $f_t$, for all $t$. Similarly, if $\psi:\Pi\to \Delta$ is the morphism induced by the lift $\tilde{G}$ of $G$, this is also the morphism induced by the lift $\tilde{g}_t$ of $g_t$, for all $t$. Consequently, the set of coincidence classes of $(f_t,g_t)$ is \[
\{p\Coin(\alpha(\tilde{f}_t)_i,\tilde{g}_t) \mid [(\alpha,i)]\in \R[\varphi,\psi] \}
\]
for all $t$.
\end{rmk}

\begin{thm}\label{thm:cc-paths}
Two points $x,x'\in \Coin(f,g)$ belong to the same coincidence class if and only if there is a path $c:I\to X$ from $x$ to $x'$ such that, if $fc=\{c_1,\ldots,c_n \}$, there is an $i\in \{1,\ldots,n\}$ such that the paths $c_i$ and $gc$ are homotopic.
\end{thm}

\begin{proof}
First, suppose two points $x$ and $x'$ belong to the same coincidence class $p\Coin(\alpha\tilde{f}_i,\tilde{g})$; say $x=p(\tilde{x})$ and $x'=p(\tilde{x}')$ with $\tilde{x},\tilde{x}'\in \Coin(\alpha\tilde{f}_i,\tilde{g})$. Let $\tilde{c}:I\to \tilde{X}$ be a path from $\tilde{x}$ to $\tilde{x}'$. Then $\alpha\tilde{f}_i\tilde{c}$ and $\tilde{g}\tilde{c}$ are paths in $\tilde{Y}$ with the same begin and endpoint. Since $\tilde{Y}$ is simply connected, they are homotopic.

Define $c=p\circ \tilde{c}$. This is a path in $X$ from $x$ to $x'$, with \[
fc=fp\tilde{c}=q\alpha\tilde{f}\tilde{c}=\{q\alpha\tilde{f}_1\tilde{c},\ldots,q\alpha\tilde{f}_n\tilde{c} \}
\]
and $gc=q\tilde{g}\tilde{c}$. Since $\alpha\tilde{f}_i\tilde{c}$ and $\tilde{g}\tilde{c}$ are homotopic paths in $\tilde{Y}$, the paths $q\alpha\tilde{f}_i\tilde{c}$ and $q\tilde{g}\tilde{c}$ are homotopic in $Y$.

Conversely, for $x,x'\in \Coin(f,g)$, suppose $c:I\to X$ is a path from $x$ to $x'$ with $fc=\{c_1,\ldots,c_n \}$, where $c_i$ is homotopic to $gc$. Let $p\Coin(\alpha\tilde{f}_i,\tilde{g})$ be the coincidence class containing $x$, and choose $\tilde{x}\in \Coin(\alpha\tilde{f}_i,\tilde{g})$ with $p(\tilde{x})=x$. There is a unique lift $\tilde{c}:I\to \tilde{X}$ of $c$ with $\tilde{c}(0)=\tilde{x}$. Define $\tilde{x}'=\tilde{c}(1)$; then $p(\tilde{x}')=x'$. 

Note that $q\alpha\tilde{f}_i\tilde{c}(t)\in fp\tilde{c}(t)=fc(t)$ for all $t\in I$, so $q\alpha\tilde{f}_i\tilde{c}$ is one of the factors of $fc$; since $q\alpha\tilde{f}_i\tilde{c}(0)=x=c_i(0)$, it must be $c_i$. Thus, $\alpha\tilde{f}_i\tilde{c}$ is the lift of $c_i$ starting at $\alpha\tilde{f}_i(\tilde{x})$. On the other hand, $\tilde{g}\tilde{c}$ is the lift of $gc$ starting at $\tilde{g}(\tilde{x})$. Since $c_i$ and $gc$ are homotopic, the endpoints of their lifts starting at $\alpha\tilde{f}_i(\tilde{x})=\tilde{g}(\tilde{x})$ must be equal. Thus, $\alpha\tilde{f}_i(\tilde{x}')=\tilde{g}(\tilde{x}')$, which means $x'$ lies in $p\Coin(\alpha\tilde{f}_i,\tilde{g})$, the same coincidence class as $x$.
\end{proof}

\begin{thm}\label{thm:cc-open-compact}
Coincidence classes are compact and open in $\Coin(f,g)$.
\end{thm}

\begin{proof}
To show that the coincidence classes are open, take $x\in \Coin(f,g)$. We will construct an open neighborhood $U\subseteq X$ of $x$ such that $U\cap \Coin(f,g)$ is contained in the coincidence class of $x$.

Write $f(x)=\{y_1,\ldots,y_n \}$, and let $V_1,\ldots,V_n$ be disjoint open neighborhoods of the points $y_1,\ldots,y_n$. Then there is an open neighborhood of $x$ on which $f$ splits as $\{f_1,\ldots,f_n \}$, where the image of $f_i$ is contained in $V_i$ for each $i$ (see \cite[Proposition 2.2]{browngoncalves}). Now, let $i$ be such that $f_i(x)=g(x)$. Take a simply connected neighborhood $W$ of $g(x)$ contained in $V_i$. Lastly, let $U\subseteq f_i^{-1}(W)\cap g^{-1}(W)$ be a path connected neighborhood of $x$.

Consider $x'\in U\cap \Coin(f,g)$. Since $g(x')\in W\subseteq V_i$ and the only element of $f(x')=\{f_1(x'),\ldots,f_n(x') \}$ lying in $V_i$ is $f_i(x')$, we have $g(x')=f_i(x')$. Let $c$ be a path in $U$ from $x$ to $x'$. Since $g(x)=f_i(x)$ and $g(x')=f_i(x')$, the paths $gc$ and $f_ic$ in $W$ have the same begin and endpoints. Since $W$ is simply connected, they are homotopic. By Theorem \ref{thm:cc-paths}, $x$ and $x'$ lie in the same coincidence class.

Since $\Coin(f,g)$ is a disjoint union of the coincidence classes, and each coincidence class is open in $\Coin(f,g)$, each coincidence class is also closed in $\Coin(f,g)$ (since its complement, the union of the other coincidence classes, is open). Since $\Coin(f,g)$ is closed in $X$ (Lemma \ref{lem:coin-closed}), each coincidence class is closed in $X$, hence compact since $X$ is compact.
\end{proof}

\begin{cor}\label{cor:cc-finite}
There are only finitely many non-empty coincidence classes.
\end{cor}

\begin{proof}
The coincidence classes form an open cover of $\Coin(f,g)$. Since $\Coin(f,g)$ is compact, this cover admits a finite subcover.
\end{proof}

\begin{rmk}\label{rmk:fix-and-1valued-cc}
In the single-valued case and in the setting of fixed points of $n$-valued maps, the definitions of coincidence classes using paths (as in the introduction) are special cases of Theorem \ref{thm:cc-paths}. Also our algebraic definition of coincidence classes reduces to the classical one in the setting of fixed points of $n$-valued maps and in the single-valued case (as defined respectively in \cite{charlotte1} and e.g.\ in \cite{goncalves-htfpt}):
	
If $X=Y$ and $g$ is the identity map on $X$, then $\tilde{g}$ is the identity map on $\tilde{X}$, the morphism $\psi:\Pi\to \Pi$ induced by $g$ is the identity morphism, and the relation $\sim_{\varphi,\psi}$ on $\Pi\times \{1,\ldots,n\}$ is given by \begin{align*}
(\alpha,i)\sim_{\varphi,\psi} (\beta,j) &\iff \exists \gamma\in \Pi: i=\sigma_\gamma(j),\; \alpha=\gamma\beta\varphi_j(\gamma^{-1}) \\
&\iff (\alpha,i)\sim_{\varphi} (\beta,j),
\end{align*}
so the coincidence classes coincide with the $n$-valued fixed point classes of $f$.

If $f:X\to Y$ is a single-valued map, its induced morphism in the $n$-valued sense is $\varphi=(\varphi_1,1):\Pi\to \Delta\rtimes \S_1$ where $\tilde{f}\gamma=\varphi_1(\gamma)\tilde{f}$ for all $\gamma\in \Pi$. Thus, $\varphi_1:\Pi\to \Delta$ is the morphism induced by $f$ in the single-valued sense. The relation $\sim_{\varphi,\psi}$ on $\Delta\times \{1\}$ is given by \begin{align*}
(\alpha,1)\sim_{\varphi,\psi} (\beta,1) &\iff \exists \gamma\in \Pi: 1=1,\; \alpha=\psi(\gamma)\beta\varphi_1(\gamma^{-1}) \\
&\iff \alpha\sim_{\varphi_1,\psi} \beta,
\end{align*}
so the coincidence classes coincide with the single-valued coincidence classes of $(f,g)$.
\end{rmk}

\subsection{The splitting space of $f$}\label{subsec:S}

Let $S=\ker\sigma$, where $\sigma:\Pi\to \S_n$ is the morphism induced by some lift $\tilde{f}$ of $f$. Note that $S$ does not depend on the chosen lift, since the morphism $\Pi\to \S_n$ induced by another lift is $\eta\sigma\eta^{-1}$ for some $\eta\in \S_n$, which has the same kernel as $\sigma$ (see Remark \ref{rmk:lift-choice}). Note that $S$ is a finite index normal subgroup of $\Pi$.

\begin{rmk}
The group $S$ can also be defined as follows. Let \[
f_*:\pi_1(X)\cong\Pi\to \pi_1(D_n(Y))\cong B_n(X)
\] 
be the fundamental group morphism of $f$. Here $B_n(Y)$ is the braid group on $n$ strands in $Y$. Let $\rho:B_n(Y)\to \S_n$ be the morphism sending a braid to its underlying permutation (whose kernel is the pure braid group $P_n(Y)$). Then $S=\ker(\rho\circ f_*)$. In particular, the group $S$ only depends on the homotopy class of $f$.
\end{rmk}

Let $\hat{X}=S\orb \tilde{X}$ be the finite covering space of $X$ corresponding to $S$, with covering map $\hat{p}:\hat{X}\to X$. Write $\hat{f}=f\circ \hat{p}$ and $\hat{g}=g\circ \hat{p}$. \[
\begin{tikzcd}
\hat{X} \ar[d,"\hat{p}"'] \ar[dr,"\hat{f}"] &  & \hat{X} \ar[d,"\hat{p}"'] \ar[dr,"\hat{g}"] & \\
X \ar[r,"f"{xshift=-1pt}] & D_n(Y) & X \ar[r,"g"{xshift=-2pt}] & Y.
\end{tikzcd}
\]

\begin{lem}
The $n$-valued map $\hat{f}:\hat{X}\to D_n(Y)$ is split.
\end{lem}

\begin{proof}
Let $\tilde{f}$ be a lift of $f$, and let $\check{p}:\tilde{X}\to \hat{X}$ denote the universal covering map, so that we have a commutative diagram \[
\begin{tikzcd}
\tilde{X} \ar[r,"\tilde{f}"] \ar[d,"\check{p}"] \ar[dd,bend right,"p"'] & F_n(\tilde{Y},\Delta) \ar[dd,"q^n"] \\
\hat{X} \ar[d,"\hat{p}"] \ar[dr,"\hat{f}"] & \\
X \ar[r,"f"] & D_n(Y).
\end{tikzcd}
\]
That is, $\hat{f}(\check{p}(\tilde{x}))=\{q(\tilde{f}_1(\tilde{x})),\ldots,q(\tilde{f}_n(\tilde{x})) \}$ for all $\tilde{x}\in\tilde{X}$. It suffices to prove that the assignment $\check{p}(\tilde{x})\mapsto q(\tilde{f}_i(\tilde{x}))$ determines a well-defined map $\hat{f}_i:\hat{X}\to Y$, for all $i\in\{1,\ldots,n\}$.

Suppose $\check{p}(\tilde{x})=\check{p}(\tilde{x}')$. Then there is a $\gamma\in S$ so that $\tilde{x}'=\gamma\tilde{x}$. Since $S=\ker\sigma$, we have $\sigma_\gamma(i)=i$ for all $i$, so \[
q(\tilde{f}_i(\tilde{x}'))=q(\tilde{f}_i(\gamma\tilde{x}))=q(\varphi_i(\gamma)\tilde{f}_i(\tilde{x}))=q(\tilde{f}_i(\tilde{x}))
\]
for all $i$.
\end{proof}

\begin{df}
We call $\hat{X}$ the \emph{splitting space} of $f$.
\end{df}

Since the group $S$ is homotopy-invariant, so is the splitting space $\hat{X}$.

\begin{rmk}
Consider the finite covering map \[
r:F_n(Y)\to D_n(Y):(y_1,\ldots,y_n)\mapsto \{y_1,\ldots,y_n \}.
\]
In the spirit of the previous remark, we can also prove that $\hat{f}$ is split by observing that it admits a lift to $F_n(Y)$; indeed \[
\hat{f}_*(\pi_1(\hat{X}))=f_*(\hat{p}_*(\pi_1(\hat{X})))\subseteq f_*(S)\subseteq \ker\rho=r_*(P_n(Y))=r_*(\pi_1(F_n(Y))).
\]
Thus, there exists a map $\hat{X}\to F_n(X)$ making the following diagram commute: \[
\begin{tikzcd}
& F_n(Y) \ar[d,"r"] \\
\hat{X} \ar[r,"\hat{f}"{xshift=1pt}] \ar[ur,yshift=2pt] & D_n(Y).
\end{tikzcd}
\]
Since $F_n(Y)\subseteq Y^n$, this map can be written as $(\hat{f}_1,\ldots,\hat{f}_n)$ for single-valued maps $\hat{f}_i:\hat{X}\to Y$, which are then precisely the maps such that $\hat{f}$ splits as $\{\hat{f}_1,\ldots,\hat{f}_n\}$.
\end{rmk}

Coincidences of $(f,g)$ naturally relate to coincidences of the pairs $(\hat{f}_i,\hat{g})$: indeed, \begin{align*}
\hat{p}(\hat{x})\in \Coin(f,g) &\iff g(\hat{p}(\hat{x}))\in f(\hat{p}(\hat{x})) \\
&\iff \hat{g}(\hat{x}) \in \hat{f}(\hat{x})=\{\hat{f}_1(\hat{x}),\ldots,\hat{f}_n(\hat{x}) \} \\
&\iff \exists i\in \{1,\ldots,n\}: \hat{g}(\hat{x})=\hat{f}_i(\hat{x}).
\end{align*}
Since all images $\hat{f}_i(\hat{x})$ are distinct, we can write \[
\hat{p}^{-1}(\Coin(f,g)) = \bigsqcup_{i=1}^n \Coin(\hat{f}_i,\hat{g}).
\]
We will define the coincidence index of $(f,g)$ in terms of the single-valued coincidence indices of the pairs $(\hat{f}_i,\hat{g})$ (as defined e.g.\ in \cite{goncalves-htfpt}).

\subsection{The coincidence index}\label{subsec:ind}

\begin{df}
For maps $f:X\to D_n(Y)$ and $g:X\to Y$ and an open subset $U\subseteq X$, we say $(f,g,U)$ is \emph{admissible} if \[
\Coin(f,g,U) \vcentcolon= \Coin(f,g)\cap U
\]
is compact. We say a pair of homotopies $(F,G)=\{(f_t,g_t)\mid t\in I\}$ is \emph{admis\-sible in $U$} if $(F,G,U\times I)$ is admissible, i.e.\ if \[
\Coin(F,G,U\times I)=\{(x,t)\in U\times I \mid g_t(x)\in f_t(x) \}
\]
is compact. In that case we also say the triples $(f_0,g_0,U)$ and $(f_1,g_1,U)$ are \emph{admissibly homotopic}.
\end{df}

\begin{rmk}\label{rmk:def-admissible-homotopy}
In \cite{goncalves-htfpt}, the definition of an admissible homotopy (for single-valued maps) is different: a homotopy $\{(f_t,g_t)\mid t\in I\}$ is called admissible in $U$ if the set \[
K=\bigcup_{t\in I} \Coin(f_t,g_t,U)
\]
is compact. However, in our setting, these definitions are equivalent: if $(F,G)$ is admissible in $U$ in our sense, then it is also admissible in the above sense, since $K$ is the projection on the first component of $\Coin(F,G,U\times I)$. Conversely, if $K$ is compact, then \[
\Coin(F,G,U\times I)=(K\times I)\cap \Coin(F,G)
\]
is the intersection of a compact set and a closed set, hence compact.
\end{rmk}

\begin{lem}\label{lem:admissible-lift}
If $(f,g,U)$ is an admissible triple, then so is $(\hat{f}_i,\hat{g},\hat{p}^{-1}(U))$, for all $i\in \{1,\ldots,n\}$.
\end{lem}

\begin{proof}
Since $(f,g,U)$ is admissible, $\Coin(f,g)\cap U$ is compact. Since $\hat{p}$ is a finite covering map, it is proper, so $\hat{p}^{-1}(\Coin(f,g)\cap U)$ is compact. Since \begin{align*}
\hat{p}^{-1}(\Coin(f,g)\cap U)&=\hat{p}^{-1}(\Coin(f,g))\cap \hat{p}^{-1}(U)\\&=\left(\bigsqcup_{i=1}^n \Coin(\hat{f}_i,\hat{g})\right)\cap \hat{p}^{-1}(U)\\&=\bigsqcup_{i=1}^n \left(\Coin(\hat{f}_i,\hat{g})\cap \hat{p}^{-1}(U)\right),
\end{align*}
each of the sets $\Coin(\hat{f}_i,\hat{g})\cap \hat{p}^{-1}(U)$ is the intersection of the compact set $\hat{p}^{-1}(\Coin(f,g)\cap U)$ with the closed set $\Coin(\hat{f}_i,\hat{g})$, hence it is compact as well.
\end{proof}

\begin{df}
For any admissible triple $(f,g,U)$, we define \[
\ind(f,g,U)=\frac{1}{[\Pi:S]} \sum_{i=1}^{n} \ind_1(\hat{f}_i,\hat{g},\hat{p}^{-1}(U))
\]
where $\ind_1$ denotes the single-valued coincidence index.
\end{df}

\begin{thm}\label{thm:index}
The coincidence index satisfies the following properties:
\begin{itemize}
\item[\rm(i)] {\sc Homotopy invariance} 

\noindent
If $\{(f_t,g_t)\mid t\in I \}$ is an admissible homotopy in $U$, then \[
\ind(f_0,g_0,U)=\ind(f_1,g_1,U).
\]

\item[\rm(ii)] {\sc Additivity} 

\noindent
If $U_1,U_2\subseteq U$ are disjoint opens such that $\Coin(f,g,U)\subseteq U_1\cup U_2$, then \[
\ind(f,g,U)=\ind(f,g,U_1)+\ind(f,g,U_2).\footnote{Note that when $(f,g,U)$ is admissible, so are $(f,g,U_1)$ and $(f,g,U_2)$, since $\Coin(f,g,U_1)$ and $\Coin(f,g,U_2)$ are both open and each other's complement, and hence both closed, in the compact set $\Coin(f,g,U)$.}
\]

\item[\rm(iii)] {\sc Splitting property} 

\noindent
If $f|_U=\{f_1,\ldots,f_n\}$, then \[
\ind(f,g,U)=\sum_{i=1}^n \ind_1(f_i,g,U).
\]
\end{itemize}
\end{thm}

Before proving the theorem, we note that it has the following immediate consequences:

\begin{cor}\label{cor:index}
The coincidence index satisfies the following properties:
\begin{itemize}
\item[\rm(iv)] {\sc Excision}

\noindent
If $V\subseteq U$ is open and $\Coin(f,g,U)\subseteq V$, then \[
\ind(f,g,U)=\ind(f,g,V).\footnote{Note that when $(f,g,U)$ is admissible, so is $(f,g,V)$, as $\Coin(f,g,V)=\Coin(f,g,U)$.}
\]

\item[\rm(v)] {\sc Solution}

\noindent
If $\ind(f,g,U)\neq 0$, then $\Coin(f,g,U)\neq \varnothing$.
\end{itemize}
\end{cor}

\begin{proof}[Proof of Corollary \ref{cor:index}]
Note that $(f,g,\varnothing)$ is admissible. Using (ii), \[
\ind(f,g,\varnothing)=\ind(f,g,\varnothing)+\ind(f,g,\varnothing),
\]
so $\ind(f,g,\varnothing)=0$. Property (iv) follows by applying (ii) to $V$ and $\varnothing$. The contrapositive of (v) follows by applying (iv) to $V=\varnothing$.
\end{proof}

\begin{proof}[Proof of Theorem \ref{thm:index}]
For (i), suppose $(F,G)=\{(f_t,g_t)\mid t\in I \}$ is an admissible homotopy in $U$. As we saw in Remark \ref{rmk:cc-homotopy}, the morphism $\sigma$ induced by $F$ is also the one induced by all maps $f_t$. In particular, the splitting space of $F$ is $\hat{X}\times I$ (with covering map $\hat{p}\times \id:\hat{X}\times I\to X\times I$), where $\hat{X}$ (with covering map $\hat{p}:\hat{X}\to X$) is the splitting space of all maps $f_t$. We have \[
\hat{F}=F\circ (\hat{p}\times \id):\hat{X}\times I\to D_n(Y):(\hat{x},t)\mapsto \hat{f}_t(\hat{x}).
\]
If $\hat{F}=\{\hat{F}_1,\ldots,\hat{F}_n\}$ is a splitting, then for each $t\in I$, a splitting for $\hat{f}_t$ is $\{(\hat{f}_t)_1,\ldots,(\hat{f}_t)_n \}$ with $(\hat{f}_t)_i(\hat{x})\vcentcolon=\hat{F}_i(\hat{x},t)$ for all $i$ (i.e., $\hat{F}_i=\{(\hat{f}_t)_i\mid t\in I\}$). We also have \[
\hat{G}:\hat{X}\times I\to Y: (\hat{x},t)\mapsto \hat{g}_t(\hat{x}),
\]
so $\hat{G}=\{\hat{g}_t\mid t\in I\}$.
Since $(F,G,U\times I)$ is admissible, so is \[
(\hat{F}_i,\hat{G},(\hat{p}\times \id)^{-1}(U\times I))=(\hat{F}_i,\hat{G},\hat{p}^{-1}(U)\times I)
\] 
for each $i$, by Lemma \ref{lem:admissible-lift}. Thus, $\{((\hat{f}_t)_i,\hat{g}_t)\mid t\in I \}$ is an admissible homotopy in $\hat{p}^{-1}(U)$. By homotopy-invariance of the single-valued coincidence index, we get \[
\ind_1((\hat{f}_0)_i,\hat{g}_0,\hat{p}^{-1}(U))=\ind_1((\hat{f}_1)_i,\hat{g}_1,\hat{p}^{-1}(U))
\]
for all $i$. Consequently, \begin{align*}
\ind(f_0,g_0,U) &= \frac{1}{[\Pi:S]} \sum_{i=1}^n \ind_1((\hat{f}_0)_i,\hat{g}_0,\hat{p}^{-1}(U)) \\
&= \frac{1}{[\Pi:S]} \sum_{i=1}^n \ind_1((\hat{f}_1)_i,\hat{g}_1,\hat{p}^{-1}(U)) \\
&= \ind(f_1,g_1,U).
\end{align*}

For (ii), let $U_1,U_2\subseteq U$ be disjoint opens with $\Coin(f,g,U)\subseteq U_1\cup U_2$. Then $\hat{p}^{-1}(U_1)$ and $\hat{p}^{-1}(U_2)$ are disjoint opens contained in $\hat{p}^{-1}(U)$ with \begin{align*}
\Coin(\hat{f}_i,\hat{g},\hat{p}^{-1}(U))&\subseteq \hat{p}^{-1}(\Coin(f,g,U))\\&\subseteq \hat{p}^{-1}(U_1\cup U_2)\\&=\hat{p}^{-1}(U_1)\cup \hat{p}^{-1}(U_2).
\end{align*}
By additivity of the single-valued coincidence index, \begin{align*}
\ind(f,g,U) &= \frac{1}{[\Pi:S]} \sum_{i=1}^n \ind_1(\hat{f}_i,\hat{g},\hat{p}^{-1}(U)) \\
&= \frac{1}{[\Pi:S]} \sum_{i=1}^n \left(\ind_1(\hat{f}_i,\hat{g},\hat{p}^{-1}(U_1))+\ind_1(\hat{f}_i,\hat{g},\hat{p}^{-1}(U_2))\right) \\
&= \ind(f,g,U_1)+\ind(f,g,U_2).
\end{align*}

For (iii), suppose $f|_U=\{f_1,\ldots,f_n \}$. Then on $\hat{p}^{-1}(U)$ we can write $\hat{f}=f\circ \hat{p}=\{f_1\circ\hat{p},\ldots,f_n\circ \hat{p} \}$, so we can choose an order on the factors $\hat{f}_i$ such that $\hat{f}_i=f_i\circ \hat{p}$ on $\hat{p}^{-1}(U)$. We get \[
\ind(f,g,U) = \frac{1}{[\Pi:S]} \sum_{i=1}^n \ind_1(f_i\circ \hat{p},g\circ \hat{p},\hat{p}^{-1}(U)).
\]
Note that \[
\Coin(f,g,U)=\bigsqcup_{i=1}^n \Coin(f_i,g,U).
\]
Since $\Coin(f,g,U)$ is compact, so is $\Coin(f_i,g,U)$ for each $i$ (being the intersection of the compact set $\Coin(f,g,U)$ with the closed set $\Coin(f_i,g)$). By single-valued coincidence theory, we can find maps $f'_i,g':U\to Y$ such that $(f_i,g)$ and $(f'_i,g')$ are admissibly homotopic in $U$ and $\Coin(f'_i,g',U)$ is finite. Then \[
\ind_1(f_i,g,U)=\ind_1(f'_i,g',U)=\sum_{x\in U} \ind_1(f'_i,g',x)
\]
with $\ind_1(f'_i,g',x)$ the local coincidence index\footnote{If $U\subseteq X$ is an open neighborhood of $x$ such that $\Coin(f'_i,g',U)=\{x\}$, the index $\ind_1(f'_i,g',x)$ is defined as $\ind_1(f'_i,g',U)$, which is independent of the choice of $U$ by the excision property of the single-valued index.}, which is non-zero for finitely many $x\in U$. Since the index is a local invariant, and $\hat{p}$ is an orientation-preserving local homeomorphism, we have \[
\ind_1(f'_i,g',x)=\ind_1(f'_i\circ \hat{p},g'\circ \hat{p},\hat{x})
\] 
for all $\hat{x}\in \hat{p}^{-1}(x)$. Since $\hat{p}$ is a $[\Pi:S]$-fold cover, each $x\in U$ has $[\Pi:S]$ preimages $\hat{x}\in \hat{p}^{-1}(U)$, so \begin{align*}
\sum_{x\in U}\ind_1(f'_i,g',x)&=\frac{1}{[\Pi:S]}\sum_{\hat{x}\in \hat{p}^{-1}(U)}\ind_1(f'_i\circ \hat{p},g'\circ \hat{p},\hat{x})\\&=\frac{1}{[\Pi:S]}\ind_1(f'_i\circ \hat{p},g'\circ \hat{p},\hat{p}^{-1}(U)).
\end{align*}
Composing the homotopies between $f'_i,g':U\to Y$ and $f_i,g:U\to Y$ with $\hat{p}$ gives homotopies between the pairs $f'_i\circ \hat{p},g'\circ \hat{p}:\hat{p}^{-1}(U)\to Y$ and $f_i\circ \hat{p},g\circ \hat{p}:\hat{p}^{-1}(U)\to Y$, which are still admissible by Lemma \ref{lem:admissible-lift}. So we also have \[
\ind_1(f'_i\circ \hat{p},g'\circ \hat{p},\hat{p}^{-1}(U))=\ind_1(f_i\circ \hat{p},g\circ \hat{p},\hat{p}^{-1}(U))
\]
for all $i$. Together, we obtain \begin{align*}
\ind(f,g,U) &= \frac{1}{[\Pi:S]} \sum_{i=1}^n \ind_1(f_i\circ \hat{p},g\circ \hat{p},\hat{p}^{-1}(U)) \\
&= \sum_{i=1}^n \ind_1(f_i,g,U). \qedhere
\end{align*}
\end{proof}

\begin{rmk}\label{rmk:local-index}
Suppose $x$ is an isolated coincidence point of $(f,g)$, i.e.\ there exists a $U\subseteq X$ such that $\Coin(f,g,U)=\{x\}$. This $U$ can be chosen small enough that $f$ splits on $U$, say $f|_U=\{f_1,\ldots,f_n\}$. Then there is a unique $i\in \{1,\ldots,n\}$ such that $x\in \Coin(f_i,g)$, and \[
\ind(f,g,x)\vcentcolon=\ind(f,g,U)=\sum_{i=1}^{n}\ind_1(f_i,g,U)=\ind_1(f_i,g,x).
\]
\end{rmk}

\begin{rmk}\label{rmk:1valued-index}
If $f$ is split on the whole space $X$, say $f=\{f_1,\ldots,f_n \}$, then $S=\Pi$, $\hat{X}=X$ and $\hat{p}$ is the identity map, so we get \[
\ind(f,g,U)=\sum_{i=1}^{n} \ind_1(f_i,g,U).
\]
In particular, if $n=1$, our index coincides with the single-valued coincidence index.
\end{rmk}

\begin{rmk}\label{rmk:fix-index}
If $X=Y$ and $g$ is the identity map, then \[
\ind(f,U)\vcentcolon=\ind(f,\id,U)=\frac{1}{[\Pi:S]} \sum_{i=1}^n \ind_1(\hat{f}_i,\hat{p},\hat{p}^{-1}(U))
\]
defines a function \[
\ind:\{\text{admissible pairs }(f,U) \}\to \RR
\]
(where $(f,U)$ is admissible if $\Fix(f)\cap U$ is compact) satisfying \smallskip \begin{itemize}
\item[(i)'] If $\{f_t\mid t\in I\}$ is an admissible homotopy in $U$, then \[
\ind(f_0,U)=\ind(f_1,U).
\]
\item[(ii)'] If $U_1,U_2\subseteq U$ are disjoint opens such that $\Fix(f,U)\subseteq U_1\cup U_2$, then \[
\ind(f,U)=\ind(f,U_1)+\ind(f,U_2).
\]
\item[(iii)'] If $f$ splits as $\{f_1,\ldots,f_n \}$ on $U$ and $\Fix(f,U)$ is a single point, then \[
\ind(f,U)=\ind_1(f_i,U)
\] with $i$ such that $\Fix(f_i,U)\neq \varnothing$.
\end{itemize}
\smallskip
(where (iii)' follows by combining (iii) with the excision and empty set properties of the single-valued index).

In \cite{staecker2018}, C.~Staecker showed that the $n$-valued fixed point index is uniquely determined by the properties (i)', (ii)' and (iii)'. Thus, $\ind(f,U)$ coincides with the $n$-valued fixed point index.
\end{rmk}

\subsection{The Nielsen number}

By Theorem \ref{thm:cc-open-compact}, we can make the following definition.

\begin{df}
For a coincidence class $C$ of $(f,g)$, define \[
\ind(f,g,C)=\ind(f,g,U)
\]
where $U\subseteq X$ is an open set such that $C=\Coin(f,g)\cap U$.
\end{df}

By the excision property (Corollary \ref{cor:index} (iv)), this definition is independent of the choice of $U$.

\begin{df}
The \emph{Nielsen number} $N(f,g)$ is the number of coincidence classes with non-zero index.
\end{df}

From the previous sections, we immediately get:

\begin{cor}\label{cor:N-finite+leq}
$N(f,g)$ is finite and $N(f,g)\leq \# \Coin(f,g)$.
\end{cor}

\begin{proof}
By the solution property (Corollary \ref{cor:index} (v)), the Nielsen number is bounded above by the number of non-empty coincidence classes. It follows from Corollary \ref{cor:cc-finite} that $N(f,g)$ is finite. Since each non-empty coincidence class contains a coincidence point (and these points are all distinct as the coincidence classes are disjoint), we get $N(f,g)\leq \# \Coin(f,g)$.
\end{proof}

\begin{thm}\label{thm:hom-inv}
The Nielsen number is a homotopy invariant.
\end{thm}

\begin{proof}
Suppose $(F,G)=\{(f_t,g_t)\mid t\in I \}$ is a pair of homotopies. Recall from Remark \ref{rmk:cc-homotopy} that the set of coincidence classes of $(f_t,g_t)$ is \[
\{p\Coin(\alpha(\tilde{f}_t)_i,\tilde{g}_t) \mid [(\alpha,i)]\in \R[\varphi,\psi] \}
\]
for all $t$. We will show that for a fixed pair $(\alpha,i)$, the index \[
\ind(f_t,g_t,p\Coin(\alpha(\tilde{f}_t)_i,\tilde{g}_t))
\]
is the same for all $t$.

Write $C=p\Coin(\alpha \tilde{F}_i,\tilde{G})$ and $C_t=p\Coin(\alpha(\tilde{f}_t)_i,\tilde{g}_t)$ for all $t\in I$. Since $C$ is a coincidence class of $(F,G)$, it is open in $\Coin(F,G)$, so we can take an open set $V\subseteq X\times I$ such that $C=\Coin(F,G)\cap V$. Then \[
C_t\times \{t\}=(\Coin(f_t,g_t)\times \{t\})\cap V
\] 
for all $t\in I$ (see Figure \ref{fig:amoebe}).

Fix $t\in I$. Since $C_t\times \{t\}\subseteq V$, each point $(x,t)\in C_t\times \{t\}$ has an open neighborhood $V_x\subseteq X\times I$ that is contained in $V$. We may further assume $V_x$ is a basic open for the product topology on $X\times I$, i.e.\ $V_x=U_x\times W_x$ for open sets $U_x\subseteq X$ and $W_x\subseteq I$.
Since $C_t$ is compact, so is $C_t\times \{t\}$, hence its open cover $\{V_x \mid x\in C_t \}$ has a finite subcover; say $\{V_{x_1},\ldots,V_{x_k} \}$. Then $U_t\vcentcolon=U_{x_1}\cup \ldots \cup U_{x_k}$ and $W'_t\vcentcolon=W_{x_1}\cap \ldots \cap W_{x_k}$ are open sets such that $C_t\times \{t\}\subseteq U_t\times W'_t\subseteq V$. 

Note that it is not necessarily true that $C_s\subseteq U_t$ for all $s\in W'_t$. However, if we let $\pr:X\times I\to I$ denote the projection on the second component, the set \[
W_t\vcentcolon=W'_t\setminus \pr(C\setminus(U_t\times I))
\]
does satisfy this property: it is open since $C\setminus(U_t\times I)$ is compact, it contains $t$, and for all $s\in W_t$, if $x\in C_s$, then \[
\pr(x,s)=s\in W_t\subseteq I\setminus \pr(C\setminus(U_t\times I))
\]
so $x\in U_t$.

\begin{figure}[h!]
\includegraphics[scale=0.875]{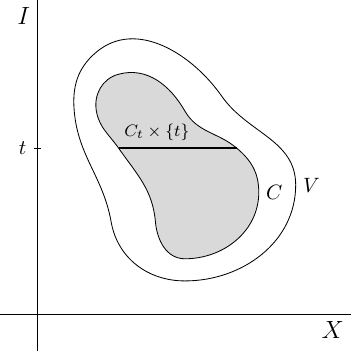}
\caption{The coincidence class $C_t$ of $(f_t,g_t)$ can be visualised as the `slice at $t$' of the coincidence class $C$ of $(F,G)$. If $V$ is an isolating neighborhood of $C$ in $\Coin(F,G)$, then the slice at $t$ of $V$ is an isolating neighborhood of $C_t$ in $\Coin(f_t,g_t)$.}
\label{fig:amoebe}
\end{figure}

\begin{figure}[h!]
\includegraphics[scale=0.875]{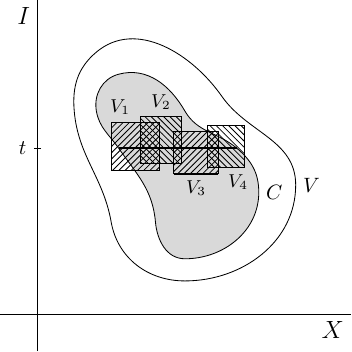}
\hfill
\includegraphics[scale=0.875]{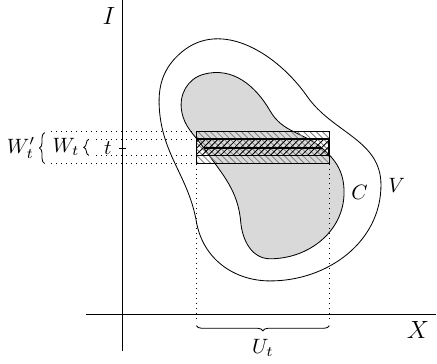}
\caption{A finite cover for $C_t\times \{t\}$ consisting of basic open subsets of $V$, and the corresponding sets $U_t$, $W'_t$ and $W_t$.}
\label{fig:amoebe2}
\end{figure}

Since $W_t\subseteq W'_t$, we also have $C_s\times \{s\}\subseteq U_t\times W'_t\subseteq V$, from which it follows that $C_s=\Coin(f_s,g_s)\cap U_t$. In particular, for every $s\in W_t$, the pair of homotopies $\{(f_r,g_r)\mid r\in [s,t]\}$ (resp.\ $\{(f_r,g_r)\mid r\in [t,s]\}$ if $s>t$) is admissible in $U_t$, since $\Coin(F,G,U_t\times[s,t])=C\cap (X\times [s,t])$, and we have \[
\ind(f_s,g_s,C_s)=\ind(f_s,g_s,U_t)=\ind(f_t,g_t,U_t)=\ind(f_t,g_t,C_t).
\]
Since the sets $W_t$ form an open cover for $I$, it follows by compactness of $I$ that $\ind(f_t,g_t,C_t)$ is the same for all $t\in I$.
\end{proof}

Together with Corollary \ref{cor:N-finite+leq}, we find that the Nielsen number is a lower bound for the number of coincidences among all pairs of maps homotopic to $(f,g)$, i.e.\ \[
N(f,g)\leq \MC(f,g)=\min\{\# \Coin(f',g') \mid f'\simeq f, g'\simeq g \}.
\]

\begin{rmk}\label{rmk:N-1v-fix}
In the single-valued setting ($n=1$), it follows from Remark \ref{rmk:fix-and-1valued-cc} and Remark \ref{rmk:1valued-index} that $N(f,g)$ is the single-valued Nielsen coincidence number of $f$ and $g$.

If $X=Y$ and $g$ is the identity map, it follows from Remark \ref{rmk:fix-and-1valued-cc} and Remark \ref{rmk:fix-index} that $N(f,g)$ coincides with the $n$-valued Nielsen fixed point number $N(f)$.
\end{rmk}

\subsection{Reidemeister and Lefschetz number}

In analogy with the single-valued case, we can also define:

\begin{df}
The \emph{Reidemeister number} $R(f,g)$ is the number of coincidence classes of the pair $(f,g)$ (also counting the empty ones).
\end{df}

Given lifts $\tilde{f}$ and $\tilde{g}$ with respective induced morphisms $\varphi$ and $\psi$ as in section \ref{subsec:cc}, we have $R(f,g)=R(\varphi,\psi)$, where $R(\varphi,\psi)$ is the number of equivalence classes for the relation $\sim_{\varphi,\psi}$.

In the single-valued setting, the following is often used as a characterisation of the Lefschetz number, but we can use it as a definition (and later, in Remark \ref{rmk:av-L}, we will deduce a characterisation more like the classical definition).

\begin{df}\label{def:L}
The \emph{Lefschetz number} of $(f,g)$ is $L(f,g)=\ind(f,g,X)$.
\end{df}

From the solution property (Corollary \ref{cor:index} (v)), we immediately get an analog of the Lefschetz fixed point theorem:

\begin{cor}
If $L(f,g)\neq 0$, then $f$ and $g$ have a coincidence point.
\end{cor}

It follows respectively from Remark \ref{rmk:cc-homotopy} and from Theorem \ref{thm:index} (i) that $R(f,g)$ and $L(f,g)$ are homotopy invariants.
By the same observations as in Remark \ref{rmk:N-1v-fix}, $R(f,g)$ and $L(f,g)$ reduce to the known invariants in the fixed point setting and in the setting of single-valued maps.

\section{Two classical results}\label{sec:axioms}

The main results of this section are the Wecken property (Theorem \ref{thm:Wecken}), and the fact that the properties of Theorem \ref{thm:index} can be taken as axioms for the index (Theorem \ref{thm:unique-index}).
For both results, we will need that any pair of maps admits an (admissible) homotopy to a pair with a finite coincidence set.

\subsection{Reduction to finitely many coincidences} 

We will show:

\begin{thm}\label{thm:coin-finite}
If $U\subseteq X$ is an open set such that $\Coin(f,g,\partial U)=\varnothing$,\footnote{Note that in that case, the triple $(f,g,U)$ is automatically admissible: $\Coin(f,g,U)=\Coin(f,g)\cap \bar{U}$ is an intersection of compact sets, hence compact.} there exists a map $g':X\to Y$ homotopic to $g$ such that $\Coin(f,g',U)$ is finite and $(f,g)$ and $(f,g')$ are admissibly homotopic in $U$.
\end{thm}

\subsubsection{Local reduction to single-valued coincidence theory}

On a sufficiently small open neighborhood of the coincidence set, we will be able to split $f$ into a single-valued and an $(n-1)$-valued part, such that all coincidences with $g$ happen in the single-valued part.

Fix a metric $\rho$ on $Y$. Following \cite{schirmer1}, we define:

\begin{df}
For an $n$-valued map $f:X\to D_n(Y)$, the \emph{gap} of $f$ is \[
\gamma(f)=\inf\{\gamma(f(x))\mid x\in X \}
\]
where \[
\gamma(\{y_1,\ldots,y_n \})=\min\{\rho(y_i,y_j)\mid i\neq j \}.
\]
\end{df}

\noindent
Note that $\gamma(f)>0$ since $X$ is compact.

We will also use the notation \[
\rho(f(x),g(x))=\min\{\rho(y,g(x)) \mid y\in f(x) \}
\]
for $x\in X$.

\begin{lem}\label{lem:f1}
Define \[
U_0=\{x\in X \mid \rho(f(x),g(x))<\tfrac{\gamma(f)}{2} \}.
\]
There exists a single-valued map $f_1:U_0\to Y$ with $f_1(x)\in f(x)$ for all $x\in U_0$, such that $\Coin(f,g)=\Coin(f_1,g)$.
\end{lem}

\begin{proof}
If $x\in U_0$, there is a $y\in f(x)$ such that $\rho(y,g(x))<\frac{\gamma(f)}{2}$. Note that all other elements $y'\in f(x)$ satisfy $\rho(y',g(x))>\frac{\gamma(f)}{2}$. Indeed, \begin{align*}
\rho(y',g(x))&\geq \rho(y',y)-\rho(y,g(x))\\ &>\rho(y',y)-\tfrac{\gamma(f)}{2}\\ &\geq \gamma(f)-\tfrac{\gamma(f)}{2}=\tfrac{\gamma(f)}{2}.
\end{align*}
Let $f_1(x)$ be the unique element $y\in f(x)$ satisfying $\rho(y,g(x))<\frac{\gamma(f)}{2}$. We will show that the resulting function $f_1:U_0\to Y$ is continuous, by showing that any $x\in U_0$ has an open neighborhood $U_x\subseteq U_0$ on which $f_1$ is continuous.

For $x\in U_0$, let $U_x\subseteq U_0$ be a connected open neighborhood of $x$ on which $f$ splits; say $f=\{f'_1,\ldots,f'_n \}$. We may choose the order of the factors so that $f'_1(x)=f_1(x)$. Since $f'_1$ is a continuous function on a connected set, its graph is connected. Also, by the observation at the beginning of this proof, \[
\text{graph}(f,U_x)\vcentcolon=\{(x,y)\in U_x\times Y\mid y\in f(x) \} \subseteq V_1 \cup V_2
\]
where \begin{align*}
V_1&=\{(x,y)\in U_x\times Y\mid \rho(y,g(x))<\tfrac{\gamma(f)}{2} \} \\
V_2&=\{(x,y)\in U_x\times Y\mid \rho(y,g(x))>\tfrac{\gamma(f)}{2} \}.
\end{align*}
Since $V_1$ and $V_2$ are disjoint open sets, and $(x,f'_1(x))\in V_1$, the graph of $f'_1$ must be entirely contained in $V_1$. Thus, for all $x'\in U_x$, we know $f'_1(x')$ is the unique point $y\in f(x)$ with $\rho(y,g(x))<\tfrac{\gamma(f)}{2}$, i.e.\ $f'_1(x')=f_1(x')$. Since $f'_1$ is continuous, it follows that $f_1$ is continuous on $U_x$.

We clearly have $\Coin(f_1,g)\subseteq \Coin(f,g)$, as $f_1(x)\in f(x)$ for all $x$. Conversely, if $g(x)\in f(x)$, then $x\in U_0$ and $g(x)$ is the unique element $y\in f(x)$ with $\rho(y,g(x))<\frac{\gamma(f)}{2}$, so $g(x)=f_1(x)$.
\end{proof}

\subsubsection{Coincidence-finite homotopies}

To prove Theorem \ref{thm:coin-finite}, we will combine Lemma \ref{lem:f1} with an approximation result due to Schirmer \cite{schirmer-coin} for coincidences of single-valued maps. 

\begin{df}
Let $M$ and $N$ be manifolds, and let $\rho$ be a metric on $N$. For $\eps>0$, two maps $g,g':M\to N$ are called \emph{$\eps$-homotopic} (with respect to the metric $\rho$) if there is a homotopy $\{g_t:M\to N\mid t\in I \}$ from $g$ to $g'$ such that $\rho(g(x),g_t(x))<\eps$ for all $x\in M$ and $t\in I$.
\end{df}

\begin{lem}[{\cite[Hilfssatz 3]{schirmer-coin}}]\label{lem:schirmer-Hilfssatz3}
Let $N$ be a closed orientable triangulable manifold equipped with a metric $\rho$.\footnote{In \cite{schirmer-coin}, this is done with respect to a metric on $N$ induced by an embedding into a Euclidean space, but the proof is independent of the choice of metric.} 
For every $\eps>0$, there exists an $\alpha(\eps)>0$ such that for any closed orientable triangulable manifold $M$, any two maps $f,g:M\to N$ with $\rho(f,g)<\alpha(\eps)$ are $\eps$-homotopic.
\end{lem}

\begin{lem}[{\cite[Satz IIa]{schirmer-coin}}]\label{lem:schirmer-SatzIIa}
Let $(f,g):M\to N$ be a pair of maps between $m$-dimensional closed triangulable orientable manifolds, $N$ equipped with a metric $\rho$, and $A\subseteq M$ a homogeneous $m$-dimensional subcomplex (i.e., for some triangulation of $M$, $A$ is a collection of $m$-simplices and their boundaries; see also \cite[p.~127]{alexandroffhopf}). Suppose $f\neq g$ on $\partial A$. For any $\tau>0$, there exists a map $g':A\to N$ with $\rho(g,g')<\tau$ on $A$ and $g=g'$ on $\partial A$, such that all coincidence points of $f$ and $g'$ in $A$ are isolated (hence $\Coin(f,g',A)$ is finite, since $A$ is compact).\footnote{This lemma can clearly also be applied to maps $f$ and $g$ defined only on $A$, rather than on the whole manifold $M$.}
\end{lem}

\begin{proof}[Proof of Theorem \ref{thm:coin-finite}]
Let $U_0$ and $f_1:U_0\to Y$ be as in Lemma \ref{lem:f1} and its proof. Since $U\cap U_0$ is open and $\Coin(f,g,U)\subseteq U\cap U_0$ is compact, one can refine the triangulation of $X$ to obtain a homogeneous $m$-dimensional subcomplex $A$ (with $m=\dim X$) such that \[
\Coin(f,g,U)\subseteq \mathring{A} \subseteq A \subseteq U\cap U_0.
\]
Since $\Coin(f,g,\bar{U}\setminus \mathring{A})=\varnothing$ and $\bar{U}\setminus \mathring{A}$ is compact, \[
\begin{split}
\rho(f,g,U\setminus A)\vcentcolon=&\,\inf\{\rho(f(x),g(x))\mid x\in U\setminus A\}\\
\geq &\,\inf\{\rho(f(x),g(x))\mid x\in \bar{U}\setminus \mathring{A}\} >0.
\end{split}
\]
Define \[
\eps=\min\left\{\tfrac{\gamma(f)}{2},\,\rho(f,g,U\setminus A)\right\}.
\]
By Lemma \ref{lem:schirmer-SatzIIa} applied to the pair $(f_1,g):A\to Y$, there exists a map $g':A\to Y$ such that $\rho(g,g')<\alpha(\eps)$ (with $\alpha(\eps)$ as in Lemma \ref{lem:schirmer-Hilfssatz3}), $g'=g$ on $\partial A$, and $\Coin(f_1,g',A)$ is finite. Extend $g'$ to $X$ by defining $g'=g$ on $X\setminus A$. By Lemma \ref{lem:schirmer-Hilfssatz3}, there is an $\eps$-homotopy $\{g_t:X\to Y \mid t\in I\}$ with $g_0=g$ and $g_1=g'$.

We show that $(f,g)$ is admissibly homotopic to $(f,g')$ in $U$ and that $\Coin(f,g',U)$ is finite.

Since $\eps\leq \rho(f,g,U\setminus A)$, we have $\Coin(f,g_t,U)\subseteq A$ for all $t\in I$. Indeed, if $x\in U\setminus A$, then \begin{align*}
\rho(f(x),g_t(x)) &\geq\rho(f(x),g(x))-\rho(g(x),g_t(x)) \\
&>\rho(f,g,U\setminus A)-\eps \geq 0.
\end{align*}
In particular, if we write $(F,G)=\{(f,g_t) \mid t\in I \}$ (i.e.\ $F(x,t)=f(x)$ and $G(x,t)=g_t(x)$ for all $x\in X$, $t\in I$), then \[
\Coin(F,G,U\times [0,1])=\Coin(F,G)\cap (A\times [0,1])
\]
is compact, being an intersection of closed sets in $X\times [0,1]$. Thus, the homotopy $\{(f,g_t) \mid t\in I \}$ is admissible in $U$.

On the other hand, since $\eps\leq \frac{\gamma(f)}{2}$, we have $\Coin(f,g',A)=\Coin(f_1,g',A)$. Indeed, if $x\in A$ then $x\in U_0$ and \[
\rho(g'(x),g(x))<\eps\leq \tfrac{\gamma(f)}{2},
\] 
so if $g'(x)\in f(x)$, then $g'(x)$ is the unique $y\in f(x)$ with $\rho(y,g(x))<\frac{\gamma(f)}{2}$, i.e.\ $g'(x)=f_1(x)$.

Together, \[
\Coin(f,g',U)=\Coin(f,g',A)=\Coin(f_1,g',A),
\]
which was finite.
\end{proof}

In the special case $U=X$, we get:

\begin{cor}\label{cor:coin-finite}
For any $n$-valued map $f:X\to D_n(Y)$ and for any map $g:X\to Y$, there exists a map $g':X\to Y$ homotopic to $g$ such that $\Coin(f,g')$ is finite.
\end{cor}

\subsection{The Wecken property}

We will show:

\begin{thm}\label{thm:Wecken}
Let $X$ and $Y$ be closed orientable triangulable manifolds of equal dimension $m\geq 3$. For any two maps $f:X\to D_n(Y)$ and $g:X\to Y$, there are maps $f':X\to D_n(Y)$ homotopic to $f$ and $g':X\to Y$ homotopic to $g$ with $\# \Coin(f',g')=N(f,g)$. In other words, $N(f,g)=\MC[f,g]$.
\end{thm}

For the remainder of this section, let $X$, $Y$, $f$ and $g$ be as in the theorem, fix a metric $\rho$ on $Y$, and suppose $p,q\in X$ are two isolated coincidence points in the same coincidence class of $(f,g)$. By Theorem \ref{thm:cc-paths}, there is a path $c:I\to X$ from $p$ to $q$ such that $fc=\{c_1,\ldots,c_n\}$, where one of the paths $c_i$ is homotopic to $gc$. By reordering the factors, we may assume $i=1$.

\begin{lem}
There is a path $c':I\to X$ homotopic to $c$, a neighborhood $\bar{U}$ of $c'(I)$ with $\Coin(f,g,\bar{U})=\{p,q\}$, and a homeomorphism $\varphi:\bar{U}\to \bar{V}\subseteq \RR^m$ such that $\varphi(c'(I))$ is a line segment in $\RR^m$ and $\bar{V}$ is an $\eps$-neighborhood of that line segment.
\end{lem}

\begin{proof}
For $n=1$ this is \cite[Hilfssatz 7]{schirmer-coin}. The $n$-valued case is analogous.
\end{proof}

Since the property of $gc$ being homotopic to $c_1$ is invariant under a homotopy of the path $c$, we may replace $c$ by $c'$ (but we keep writing $c$). By reparametrizing if necessary, we may further assume that $c$ is an embedding. 

From $\Coin(f,g,\bar{U})=\{p,q\}$, it follows in particular that $\Coin(gc,c_1)=\{0,1\}$ and $\Coin(gc,c_i)=\varnothing$ for all $i\geq 2$. Since $I$ is compact, there is an $\eta>0$ such that $\rho(gc,c_i)>\eta$ for all $i\geq 2$.

\medskip

The rest of the proof roughly goes as follows: we first want to deform $f$ into a map $f'$, under which $c_1$ becomes a path $\omega$ very close to $gc$. Lemma \ref{lem:f-to-omega} assures that this can be done without changing the coincidence set (here the assumption $m\geq 3$ is essential). When $gc$ and $\omega$ (now a factor of $f'c$) lie close together, we can find a small neighborhood of $c$ where $g$ and a factor of $f'$ lie close together, on which we can apply a result from Schirmer's single-valued coincidence theory (Lemma \ref{lem:schirmer-SatzVIIIa}) to deform to either no coincidences or one coincidence point with non-zero index (without changing the maps on the rest of $X$). By applying this to every pair of coincidence points in the same class, we reduce each class to either an empty set or a single point with non-zero index.

\begin{figure}[h!]
\begin{tikzpicture}[scale=2]
\draw (0,-0.1) edge[bend right] node[below]{$gc$} (1,0.1);
\draw (0,-0.1) edge[bend left=60] node[above]{$c_1\subseteq fc$} (1,0.1);

\draw[->] (1.5,0) -- node[above]{$f\simeq f'$} node[below]{$\substack{\text{without} \\ \text{changing} \\ \Coin(f,g)}$} (2,0);

\draw[line width=3mm,line cap=round,gray!50!white] (2.525,-0.08) edge[bend right] (3.475,0.11);
\draw (2.5,-0.1) edge[bend right] node[below, yshift=-0.5mm]{$gc$} (3.5,0.1);
\draw (2.5,-0.1) to[out=60,in=165] (2.6,-0.08) to[out=-15,in=215] node[above,xshift=-2mm,yshift=0.5mm]{$\omega\subseteq f'c$} (3.4,0.08) to[out=35,in=120] (3.5,0.1);

\draw[->] (4,0) -- node[above]{$g\simeq g'$} node[below]{$\substack{\text{on small} \\ \text{neighborhood}}$} (4.5,0);

\draw[line width=3mm,line cap=round,gray!50!white] (5.025,0.4) edge[bend right] (5.975,0.6);
\draw (5,0.38) edge[bend right] node[below]{$g'c$} (6,0.58);
\draw (5.01,0.42) edge[bend right] node[above]{$\omega$} (5.97,0.62);

\draw (4.9,0.1) node{or};

\draw[line width=3mm,line cap=round,gray!50!white] (5.025,-0.2) edge[bend right] (5.975,0);
\draw (5.02,-0.175) edge[bend right] node[below right,yshift=1mm]{$g'c$} (6,-0.02);
\draw (5.01,-0.225) edge[bend right=25] node[above left,xshift=-3mm]{$\omega$} (5.97,0.025);
\end{tikzpicture}
\caption{Idea of the proof of Theorem \ref{thm:Wecken}.}
\end{figure}
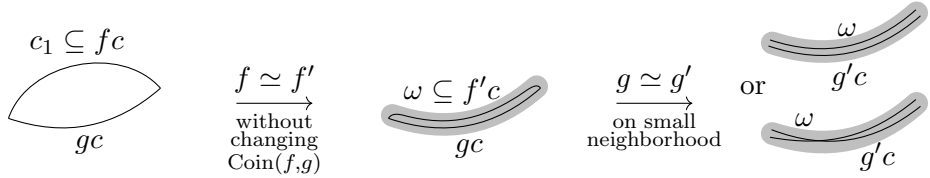

\subsubsection{Main lemma}

With the above notations, we will show:

\begin{lem}\label{lem:f-to-omega}
For any path $\omega:I\to Y$ homotopic to $gc$ with $\Coin(\omega,gc)=\{0,1\}$ and $\Coin(\omega,c_i)=\varnothing$ for all $i\geq 2$, there is a map $f':X\to D_n(Y)$ homotopic to $f$ with $\Coin(f',g)=\Coin(f,g)$, such that $f'c=\{\omega,c_2,\ldots,c_n\}$.
\end{lem}

We use the following lemma due to Schirmer. (Note that in our setting, manifolds are always closed, so one may replace $\mathring{N}$ by $N$.)

\begin{lem}[Coincidence Lemma {\cite[Lemma 2.2]{schirmer3}}]
Let $N$ be a compact manifold of dimension $m\geq 3$ equipped with a metric $\rho$, and $P$ a compact polyhedron of dimension $p<m$, with subpolyhedra $P_0\subseteq P_1\subseteq P$. For any $\eps>0$ and for any pair of maps $f,g:P\to \mathring{N}$ such that \[
f=g \text{ on } P_0,\quad f\neq g \text{ on } P_1\setminus P_0,
\]
there exists a map $f':P\to \mathring{N}$ such that \begin{itemize}
\item $f'=f$ on $P_1$
\item $f'\neq g$ on $P\setminus P_0$
\item $\rho(f,f')<\eps$.
\end{itemize}
\end{lem}

We also use the following, which is an extension of Jiang's Special Homotopy Extension Property \cite[Lemma 2.1]{jiang1980}.

\begin{lem}\label{lem:special-homotopy-extension}
Let $X$ and $Y$ be closed triangulable manifolds and $A\subseteq X$ a closed subpolyhedron. Given maps $f:X\to D_n(Y)$ and $g:X\to Y$, suppose $F:A\times I\to D_n(Y)$ is a homotopy with $F(-,0)=f|_A$ such that $\Coin(F(-,t),g|_A)=\Coin(f|_A,g|_A)$ for all $t\in I$. Suppose moreover that for all $x\in \Coin(f|_A,g|_A)$, the set $F(x,t)$ is independent of $t$.\footnote{In the single-valued case, this is automatic from $\Coin(F(-,t),g|_A)=\Coin(f|_A,g|_A)$, but when $F$ is $n$-valued we need to add this extra condition.} Then $F$ can be extended to a homotopy $\tilde{F}:X\times I\to D_n(Y)$ with $\Coin(\tilde{F}(-,t),g)=\Coin(f,g)$ for all $t\in I$.
\end{lem}

\begin{proof}
According to \cite[p.~31, O.]{hu}, there exists a retraction $r:X\times I\to X\times \{0\}\cup A\times I$. Therefore, the map \[
\tilde{F}_0:X\times \{0\}\cup A\times I\to D_n(Y):\begin{cases}
(x,0)\mapsto f(x) \\
(x,t)\mapsto F(x,t) & \text{for } x\in A.
\end{cases}
\]
extends to a homotopy $\tilde{F}'=\tilde{F}_0\circ r:X\times I\to D_n(Y)$ satisfying $\tilde{F}'(-,0)=f$ and $\tilde{F}'(x,t)=F(x,t)$ for all $x\in A$.

Now consider the set \[
C=\{x\in X\mid g(x)\in \tilde{F}'(\{x\}\times I)\}.
\]
We show that $C$ is closed, by showing that $X\setminus C$ is open. Fix metrics $d$ on $X$ and $\rho$ on $Y$, and consider the metric $d'((x_1,t_1),(x_2,t_2))=d(x_1,x_2)+|t_1-t_2|$ on $X\times I$. For $x\in X\setminus C$, let $K\subseteq X$ be a simply connected compact neighborhood of $x$. Let $\{\tilde{F}'_1,\ldots,\tilde{F}'_n\}$ be a splitting of $\tilde{F}'$ on $K\times I$. Since $x\notin C$ and $Y$ is Hausdorff, there is an $\eps>0$ such that the $\eps$-neighborhood of $g(x)$ and the $\eps$-neighborhood of the compact set $\cup_{i=1}^n \tilde{F}'_i(\{x\}\times I)$ are disjoint. Since $K\times I$ is compact, each map $\tilde{F}'_i$ is uniformly continuous, so there is a $\delta>0$ such that on $K\times I$, \[
d'((x_1,t_1),(x_2,t_2))<\delta \enskip\Rightarrow\enskip \rho(\tilde{F}'_i(x_1,t_1),\tilde{F}'_i(x_2,t_2))<\eps
\]
for all $i\in\{1,\ldots,n\}$. Also, $g$ is uniformly continuous on $K$, so there is a $\delta'>0$ such that on $K$, \[
d(x_1,x_2)<\delta' \enskip\Rightarrow\enskip \rho(g(x_1),g(x_2))<\eps.
\]
Then every $x'\in K$ with $d(x,x')<\min\{\delta,\delta'\}$ lies in $X\setminus C$: suppose $g(x')\in \tilde{F}'(\{x'\}\times I)$, then $g(x')=\tilde{F}'_i(x',t)$ for some $i$ and $t$. But $d'((x,t),(x',t))<\delta$, so $\tilde{F}'_i(x',t)$ belongs to the $\eps$-neighborhood of $\cup_{i=1}^n \tilde{F}'_i(\{x\}\times I)$, whereas $g(x')$ belongs to the $\eps$-neighborhood of $g(x)$ since $d(x,x')<\delta'$. This cannot be true since those neighborhoods were disjoint.

Since $C$ is closed, we may define \[
\tilde{F}:X\times I\to D_n(Y):(x,t)\mapsto \begin{cases}
\tilde{F}'(x,0) & \text{if } d(x,A)\geq d(x,C) \\
\tilde{F}'(x,t(1-\tfrac{d(x,A)}{d(x,C)})) & \text{if } d(x,A)<d(x,C).
\end{cases}
\]
Let us show this map is continuous. This is obvious everywhere but in the points $(x,t)$ with $d(x,A)=d(x,C)=0$, i.e.\ $x\in A\cap C$.

First, we observe that if $x\in A\cap C$, then $\tilde{F}'(x,t)$ is independent of $t$. Indeed, if $x\in A\cap C$, then $g(x)\in \tilde{F}'(\{x\}\times I)=F(\{x\}\times I)$. Take $t\in I$ such that $g(x)\in F(x,t)$. Then $x\in \Coin(F(-,t),g|_A)=\Coin(f|_A,g|_A)$, hence by our assumption, $F(x,t)=\tilde{F}'(x,t)$ is independent of $t$.

Take $(x,t)$ with $x\in A\cap C$. Let $K\subseteq X$ be a simply connected compact neighborhood of $x$. We will show that $\tilde{F}|_{K\times I}$ is continuous at $(x,t)$, by showing that for any $\eps>0$, there exists a $\delta>0$ such for $(x',t')\in K\times I$, \[
d'((x,t),(x',t'))<\delta \enskip\Rightarrow\enskip \rho_H(\tilde{F}(x,t),\tilde{F}(x',t'))<\eps,
\] 
where \[
\rho_H(\{y_1,\ldots,y_n\},\{y'_1,\ldots,y'_n\})=\max\left\{\max_{i=1}^n\min_{j=1}^n\rho(y_i,y'_j),\max_{j=1}^n\min_{i=1}^n\rho(y_i,y'_j) \right\}
\] 
is the Hausdorff distance (see \cite[Proposition 4.2]{browngoncalves}).

Let $\{\tilde{F}'_1,\ldots,\tilde{F}'_n\}$ be a splitting of $\tilde{F}'$ on $K\times I$, and choose $\eps>0$. Take $\delta>0$ such that on $K\times I$, \[
d'((x_1,t_1),(x_2,t_2))<\delta \enskip\Rightarrow\enskip \rho(\tilde{F}'_i(x_1,t_1),\tilde{F}'_i(x_2,t_2))<\eps
\]
for all $i\in\{1,\ldots,n\}$.
Take $(x',t')\in K\times I$ with $d'((x,t),(x',t'))<\delta$. Let $t''\in I$ be such that $\tilde{F}(x',t')=\tilde{F}'(x',t'')$. Then $\tilde{F}'(x,t'')=\tilde{F}'(x,0)=\tilde{F}(x,t)$ since $\tilde{F}'(x,t)$ is independent of $t$. Also, \[
d'((x,t''),(x',t''))=d(x,x')\leq d'((x,t),(x',t'))<\delta,
\]
so $\rho(\tilde{F}'_i(x,t''),\tilde{F}'_i(x',t''))<\eps$ for all $i$. 
It follows that \[
\begin{split}
\rho_H(\tilde{F}(x,t),\tilde{F}(x',t'))&=\rho_H(\tilde{F}'(x,t''),\tilde{F}'(x',t''))\\&\leq \max_{i=1}^n\rho(\tilde{F}'_i(x,t''),\tilde{F}'_i(x',t''))<\eps.
\end{split}
\]

To show that $\Coin(\tilde{F}(-,t),g)=\Coin(f,g)$ for all $t\in I$, suppose $g(x)\in \tilde{F}(x,t)$. Then $d(x,A)\geq d(x,C)$, for otherwise $d(x,C)>0$, which means $g(x)\notin \tilde{F}'(\{x\}\times I)$ so certainly $g(x)$ is not contained in $\tilde{F}(\{x\}\times I)$. But if $d(x,A)\geq d(x,C)$, then $g(x)\in \tilde{F}'(x,0)=F_0(x,0)=f(x)$. Conversely, if $g(x)\in f(x)=\tilde{F}'(x,0)$, then $d(x,C)=0$ so $\tilde{F}(x,t)=\tilde{F}'(x,0)=f(x)$.
\end{proof}

\begin{proof}[Proof of Lemma \ref{lem:f-to-omega}]
Since $\omega$ is homotopic to $gc$ and $gc$ is homotopic to $c_1$, there is a homotopy $H_1:I\times I\to Y$ from $c_1$ to $\omega$.

For every $i\geq 2$, we apply the Coincidence Lemma to $P=I\times I$, $P_0=\varnothing$, $P_1=\partial(I\times I)$ and the maps \begin{align*}
I\times I\to Y&: (t,s)\mapsto c_i(s) \\
I\times I\to Y&: (t,s)\mapsto H_1(s,t)
\end{align*}
to obtain a map $H_i:I\times I\to Y$ satisfying \begin{itemize}
\item $H_i(s,t)=c_i(s)$ if $s=0,1$ or $t=0,1$
\item $H_i(s,t)\neq H_1(s,t)$ for all $s,t\in I$
\item $\rho(H_i(s,t),c_i(s))<\min\{\eta,\tfrac{\gamma(f)}{2}\}$ for all $s,t\in I$.
\end{itemize}
(Recall that $\eta>0$ was defined such that $\rho(gc,c_i)>\eta$ for all $i\geq 2$.) By the first bullet, $H_i$ is a path homotopy starting and ending at $c_i$. For $i\neq j$, both different from $1$, \begin{align*}
\rho(H_i(s,t),H_j(s,t)) &\geq \rho(c_i(s),c_j(s))-\rho(H_i(s,t),c_i(s))-\rho(H_j(s,t),c_j(s)) \\
&>\gamma(f)-\tfrac{\gamma(f)}{2}-\tfrac{\gamma(f)}{2}=0.
\end{align*}
Combined with the second bullet, it follows that $H=\{H_1,\ldots,H_n\}$ is $n$-valued.

Now we apply the Coincidence Lemma again, to $P=I\times I$, $P_0=\partial I\times I$, $P_1=\partial (I\times I)$ and the maps \begin{align*}
I\times I\to Y&: (t,s)\mapsto H_1(s,t) \\
I\times I\to Y&: (t,s)\mapsto gc(s).
\end{align*}
We obtain a map $H'_1:I\times I\to Y$ satisfying \begin{itemize}
\item $H'_1(s,t)=H_1(s,t)$ if $s=0,1$ or $t=0,1$
\item $H'_1(s,t)\neq gc(s)$ for all $t\in I$ and $s\neq 0,1$
\item $\rho(H_1(s,t),H'_1(s,t))<\gamma(H)$ for all $s,t\in I$.
\end{itemize}
By the first bullet, $H'_1$ is still a path homotopy from $c_1$ to $\omega$. For all $i\geq 2$, \begin{align*}
\rho(H'_1(s,t),H_i(s,t)) &\geq \rho(H_1(s,t),H_i(s,t))-\rho(H_1(s,t),H'_1(s,t)) \\
&>\gamma(H)-\gamma(H)=0,
\end{align*}
so $H'=\{H'_1,H_2,\ldots,H_n\}$ is $n$-valued. Moreover, $\Coin(H'(-,t),gc)=\{0,1\}$ for all $t$: indeed, suppose $gc(s)\in H'(s,t)$. For $i\geq 2$, \begin{align*}
\rho(gc(s),H_i(s,t)) &\geq \rho(gc(s),c_i(s))-\rho(H_i(s,t),c_i(s)) \\
&>\eta-\eta=0,
\end{align*}
so $gc(s)=H'_1(s,t)$, which implies $s\in \{0,1\}$ by the second bullet.

Thus, we have a homotopy $H':I\times I\to D_n(Y)$ from $fc=\{c_1,\ldots,c_n\}$ to the path $\{\omega,c_2,\ldots,c_n\}$, such that $\Coin(H'(-,t),gc)=\{0,1\}$ for all $t$. We can apply Lemma \ref{lem:special-homotopy-extension} to $A=c(I)$ and $F:A\times I\to D_n(Y)$ given by $F(c(s),t)=H'(s,t)$ (which is well-defined since $c$ is an embedding): indeed, $F(-,0)=f|_A$, and \[
\Coin(F(-,t),g|_A)=\{c(0),c(1)\}=\Coin(f|_A,g|_A)
\]
for all $t\in I$, and for all points in $\Coin(f|_A,g|_A)$, i.e.\ for $c(s)$ with $s\in \{0,1\}$, the set $F(c(s),t)=H'(s,t)=H(s,t)=f(c(s))$ is independent of $t$.
By Lemma \ref{lem:special-homotopy-extension} we obtain a homotopy $\tilde{F}:X\times I\to D_n(Y)$ with $\tilde{F}(c(s),t)=H'(s,t)$ for all $s,t\in I$, such that $\Coin(\tilde{F}(-,t),g)=\Coin(f,g)$ for all $t$. The map $f'=\tilde{F}(-,1)$ satisfies the desired properties.
\end{proof}

\subsubsection{Applying to a suitable $\omega$}

By working in Euclidean neighborhoods covering $gc(I)$, one readily obtains:

\begin{lem}\label{lem:omega}
For every $\tau>0$, there is a path $\omega:I\to Y$ homotopic to $gc$ with $\Coin(gc,\omega)=\{0,1\}$ and $\rho(gc,\omega)<\tau$.
\end{lem}

To choose a suitable $\tau$, we recall Lemma \ref{lem:schirmer-Hilfssatz3} and another result from Schirmer's coincidence theory for single-valued maps:

\begin{lemnum}[\ref{lem:schirmer-Hilfssatz3}]
Let $N$ be a closed orientable triangulable manifold equipped with a metric $\rho$.
For every $\eps>0$, there exists an $\alpha(\eps)>0$ such that for any closed orientable triangulable manifold $M$, any two maps $f,g:M\to N$ with $\rho(f,g)<\alpha(\eps)$ are $\eps$-homotopic.
\end{lemnum}

\begin{lem}[{\cite[Satz VIIIa]{schirmer-coin}}]\label{lem:schirmer-SatzVIIIa}
Let $M$ and $N$ be closed orientable triangulable manifolds of equal dimension $m\geq 2$, and $\rho$ a metric on $N$. For every $\eps>0$, there exists a $\beta(\eps)>0$ such that the following holds. Let $A\subseteq M$ be a strongly connected\footnote{That is: for any decomposition $A=K_1\cup K_2$ into non-empty subcomplexes, $K_1$ and $K_2$ have at least one $(m-1)$-simplex in common; see \cite[p.~189]{alexandroffhopf}. Schrimer calls this `$m$-dimensionally connected'.} homogeneous $m$-dimensional subcomplex. If $g:A\to N$ and $f_0:\partial A\to N$ are maps with $0<\rho(f_0,g|_{\partial A})<\beta(\eps)$, then $f_0$ can be extended to a map $f:A\to N$ with $\rho(f,g)<\eps$, such that either $\Coin(f,g)=\varnothing$, or $f$ and $g$ have one coincidence point with non-zero index.
\end{lem}

Take $\tau=\min\{\eta,\tfrac{\alpha}{4},\frac{\beta}{3}\}$, where $\alpha=\alpha(\tfrac{\eta}{3})$ is the constant of Lemma \ref{lem:schirmer-Hilfssatz3} (with $N=Y$), and $\beta=\beta(\frac{\alpha}{4})$ is the constant of Lemma \ref{lem:schirmer-SatzVIIIa} (with $M=X$, $N=Y$). Let $\omega$ be the resulting path from Lemma \ref{lem:omega}.

Note that, since $\tau\leq \eta$, \[
\rho(\omega,c_i)\geq\rho(gc,c_i)-\rho(gc,\omega)>0
\]
for all $i\geq 2$; that is, $\Coin(\omega,c_i)=\varnothing$ for all $i\geq 2$. Thus, we can apply Lemma \ref{lem:f-to-omega}. Let $f':X\to D_n(Y)$ be the resulting map.

\subsubsection{Removing coincidences of $(f',g)$}

\begin{lem}
There is a map $g':X\to Y$ homotopic to $g$ such that $g'=g$ outside of $\bar{U}$, and $\Coin(f',g',\bar{U})$ is either empty, or contains one coincidence point with non-zero index.
\end{lem}

\begin{proof}
Since $\bar{U}$ is homeomorphic to an $\eps$-neighborhood of a line segment in $\RR^m$, it is simply connected, so the $n$-valued map $f'$ splits on $\bar{U}$; write $f'|_{\bar{U}}=\{f'_1,\ldots,f'_n\}$. Since $f'c=\{\omega,c_2,\ldots,c_n\}$, we can order the factors so that $f'_1c=\omega$ and $f'_ic=c_i$ for all $i\geq 2$.

Fix a metric $d$ on $X$. Since $\bar{U}$ is compact, the maps $g,f'_i:\bar{U}\to Y$ are uniformly continuous. Thus, there is a $\delta>0$ such that for all $x,y\in\bar{U}$ with $d(x,y)<\delta$, \begin{align*}
\rho(g(x),g(y))&<\min\{\tfrac{\alpha}{4},\tfrac{\beta}{3},\tfrac{\eta}{3}\} \\
\rho(f'_1(x),f'_1(y))&<\min\{\tfrac{\alpha}{4},\tfrac{\beta}{3}\} \\
\rho(f'_i(x),f'_i(y))&<\tfrac{\eta}{3} \quad (i\geq 2).
\end{align*}
We can make $\bar{U}$ smaller so that it is contained in the $\delta$-neighborhood of $c(I)$.

According to Schirmer (see \cite[Proof of Hilfssatz 9]{schirmer-coin}), we may refine the triangulation of $X$ to obtain a strongly connected homogeneous $m$-dimensional subcomplex $A\subseteq U$ with $p,q\in \mathring{A}$. For all $x\in \partial A$, we can choose $y=c(s)\in c(I)$ such that $d(x,y)<\delta$. Using this, we find \begin{align*}
\rho(f'_1(x),g(x))&\leq \rho(f'_1(x),f'_1(y))+\rho(f'_1(y),g(y))+\rho(g(y),g(x)) \\
&=\rho(f'_1(x),f'_1(y))+\rho(\omega(s),gc(s))+\rho(g(y),g(x)) \\
&<\tfrac{\beta}{3}+\tau+\tfrac{\beta}{3} \leq \tfrac{\beta}{3}+\tfrac{\beta}{3}+\tfrac{\beta}{3}=\beta.
\end{align*}
On the other hand, $\rho(f'_1(x),g(x))>0$, since $\Coin(f'_1,g)\subseteq\Coin(f',g)=\Coin(f,g)$ and the only coincidences of $f$ and $g$ in $\bar{U}$ are $p$ and $q$, which do not lie on $\partial A$.

Thus, by Lemma \ref{lem:schirmer-SatzVIIIa}, there is a map $g'_0:A\to Y$, equal to $g$ on $\partial A$, satisfying $\rho(f'_1|_A,g'_0)<\frac{\alpha}{4}$, with either no coincidences with $f'_1|_A$, or one coincidence point $x$ with $\ind_1(f'_1,g',x)\neq 0$.

We can extend $g'_0$ to $X$ by defining \[
g':X\to Y:x\mapsto\begin{cases}
g'_0(x) & \text{if } x\in A \\
g(x) & \text{otherwise.}
\end{cases}
\]
For $x\in A$, 
\begin{align*}
\rho(g(x),g'(x))&=\rho(g(x),g'_0(x)) \\ &\leq \rho(g(x),f'_1(x))+\rho(f'_1(x),g'_0(x)) \\
&<\rho(g(x),f'_1(x))+\tfrac{\alpha}{4}
\end{align*}
where, with $y=c(s)\in c(I)$ such that $d(x,y)<\delta$, \begin{align*}
\rho(g(x),f'_1(x))&\leq \rho(g(x),g(y))+\rho(g(y),f'_1(y))+\rho(f'_1(y),f'_1(x)) \\
&=\rho(g(x),g(y))+\rho(gc(s),\omega(s))+\rho(f'_1(y),f'_1(x)) \\
&<\tfrac{\alpha}{4}+\tau+\tfrac{\alpha}{4}\leq \tfrac{3\alpha}{4},
\end{align*}
so $\rho(g(x),g'(x))<\alpha$. On the other hand, for $x\in X\setminus A$, \[
\rho(g(x),g'(x))=\rho(g(x),g(x))=0<\alpha.
\]
Thus, $\rho(g,g')<\alpha$, so it follows from Lemma \ref{lem:schirmer-Hilfssatz3} that $g'$ is $\tfrac{\eta}{3}$-homotopic to $g$. 

Lastly, for all $i\geq 2$, for all $x\in A$, with $y=c(s)\in c(I)$ such that $d(x,y)<\delta$, \begin{align*}
\rho(g'(x),f'_i(x))&\geq \rho(g(y),f'_i(y))-\rho(f'_i(y),f'_i(x))-\rho(g(y),g(x))-\rho(g(x),g'(x)) \\
&= \rho(gc(s),c_i(s))-\rho(f'_i(y),f'_i(x))-\rho(g(y),g(x))-\rho(g(x),g'(x)) \\
&>\eta-\tfrac{\eta}{3}-\tfrac{\eta}{3}-\tfrac{\eta}{3}=0,
\end{align*}
so $\Coin(f',g',A)=\Coin(f'_1,g',A)$. Since \[
\Coin(f',g',\bar{U}\setminus A)=\Coin(f',g,\bar{U}\setminus A)=\varnothing,
\]
it follows that $\Coin(f',g',\bar{U})=\Coin(f'_1,g',A)$. Recall that this set is either empty or consists of one point $x$ with $\ind_1(f'_1,g',x)\neq 0$. In the latter case, by Remark \ref{rmk:local-index}, $\ind(f',g',x)=\ind_1(f'_1,g',x)$ is non-zero as well.
\end{proof}

\begin{proof}[Proof of Theorem \ref{thm:Wecken}]
Let $X$, $Y$, $f$ and $g$ be as in the statement of the theorem. By Corollary \ref{cor:coin-finite}, we may assume $\Coin(f,g)$ is finite, so all coincidence points are isolated. 

Given two isolated coincidence points $p$ and $q$ in the same coincidence class of $(f,g)$, the above procedure yields a set $\bar{U}\subseteq X$ with $\Coin(f,g,\bar{U})=\{p,q\}$ and a pair $(f',g')$ homotopic to $(f,g)$ such that \[
\Coin(f',g',X\setminus \bar{U})=\Coin(f,g,X\setminus \bar{U})
\]
and $\Coin(f',g',\bar{U})$ is either empty, or contains one coincidence point with non-zero index.

Since $\Coin(f,g)$ is finite, we can apply this procedure inductively for every coincidence class, resulting in a pair $(f',g')$ homotopic to $(f,g)$, each of whose coincidence classes is either empty or consists of one point with non-zero index. The empty coincidence classes have index zero, and for the ones containing one point, the index of that coincidence class equals the index of that point, which is non-zero. Thus, the number of coincidence points of $(f',g')$ is precisely the number of coincidence classes with non-zero index, i.e.\ $N(f',g')$. Since the Nielsen number is homotopy-invariant, this is equal to $N(f,g)$.
\end{proof}

\subsection{Uniqueness and equivalent definition of the index}

As another application of Theorem \ref{thm:coin-finite}, we show that the properties of Theorem \ref{thm:index} can be taken as axioms for the coincidence index, similar to the fixed point case \cite{staecker2018}:

\begin{thm}\label{thm:unique-index}
For fixed closed orientable triangulable manifolds $X$ and $Y$ of equal dimension, and for a fixed $n$, let $\mathcal{C}_n(X,Y)$ be the set of admissible triples $(f:X\to D_n(Y),\,g:X\to Y,\,U\subseteq X)$. The coincidence index is the unique function \[
\ind:\mathcal{C}_n(X,Y) \to \RR
\]
satisfying
\begin{itemize}
\item[\rm(i)] {\sc Homotopy invariance} 

\noindent
If $\{(f_t,g_t)\mid t\in I \}$ is an admissible homotopy in $U$, then \[
\ind(f_0,g_0,U)=\ind(f_1,g_1,U).
\]

\item[\rm(ii)] {\sc Additivity} 

\noindent
If $U_1,U_2\subseteq U$ are disjoint opens such that $\Coin(f,g,U)\subseteq U_1\cup U_2$, then \[
\ind(f,g,U)=\ind(f,g,U_1)+\ind(f,g,U_2).
\]

\item[\rm(iii)] {\sc Splitting property} 

\noindent
If $f=\{f_1,\ldots,f_n\}$ on $U$, then \[
\ind(f,g,U)=\sum_{i=1}^n \ind_1(f_i,g,U).
\]
\end{itemize}
\end{thm}

\begin{proof}
Suppose $\ind$ satisfies the three axioms, and let $(f,g,U)$ be an admissible triple. Since $U$ is open and $\Coin(f,g,U)$ is compact, we can find an open set $V$ such that \[
\Coin(f,g,U)\subseteq V\subseteq \bar{V}\subseteq U,
\]
so $\Coin(f,g,U)=\Coin(f,g,V)$ and $\Coin(f,g,\partial V)=\varnothing$. By Theorem \ref{thm:coin-finite}, there exists a homotopy $(f,g)\simeq (f,g')$ that is admissible in $V$ such that $\Coin(f,g',V)$ is finite. By the excision property (Corollary \ref{cor:index} (iv), which follows from the axioms) and homotopy invariance, \[
\ind(f,g,U)=\ind(f,g,V)=\ind(f,g',V).
\]
Say $\Coin(f,g',V)=\{x_1,\ldots,x_k\}$, and choose pairwise disjoint open neighborhoods $U_1,\ldots,U_k\subseteq V$ of the points $x_1,\ldots,x_k$ such that $f$ splits on $U_j$, for each $j$. Let $\displaystyle\{f^j_1,\ldots,f^j_n\}$ 
be a splitting of $f$ on $U_j$. By additivity, \[
\ind(f,g,U)=\sum_{j=1}^{k} \ind(f,g',U_j).
\]
By the splitting property, \[
\ind(f,g,U)=\sum_{j=1}^{k}\sum_{i=1}^{n} \ind_1(f_i^j,g',U_j).
\]
Thus, the value of $\ind(f,g,U)$ is entirely determined by the properties (i)--(iii). This shows that there can be at most one function $\ind$ satisfying the properties; from Theorem \ref{thm:index} we already know that our coincidence index from section \ref{subsec:ind} is such a function; hence it is the unique one.
\end{proof}

\begin{rmk}
The above construction shows that the index is integer-valued, which was not clear a priori from our definition.
\end{rmk}

\begin{rmk}\label{rmk:ind-equiv}
The above construction can be used as an equivalent definition for the index: take an admissible homotopy $(f,g)\simeq (f,g')$ in $U$ such that $\Coin(f,g',U)$ is finite; around these coincidence points, take isolating neighborhoods $U_1,\ldots,U_k\subseteq U$ on which $f$ splits, and compute \[
\ind(f,g,U)=\sum_{j=1}^{k}\sum_{i=1}^{n} \ind_1(f_i^j,g',U_j).
\]
Note that for each $j$, there is precisely one $i_j\in \{1,\ldots,n\}$ such that $f_{i_j}^j(x_j)=g(x_j)$; therefore we can also write \[
\ind(f,g,U)=\sum_{j=1}^{k} \ind_1(f_{i_j}^j,g',x_j).
\]

This is precisely the analog of how Schirmer defined the $n$-valued fixed point index in \cite{schirmer2}. The advantage of our approach is that we do not have to show this is well-defined, i.e.\ that it is independent of the chosen homotopy and of the choice of opens $U_j$; this follows immediately by uniqueness.
\end{rmk}

\section{Averaging formulas}\label{sec:av}

In this section, we express the invariants $N(f,g)$, $R(f,g)$ and $L(f,g)$ in terms of their single-valued counterparts for the pairs $(\hat{f}_i,\hat{g})$.

\subsection{Averaging formula for the Lefschetz number}

From the definition of the index and the Lefschetz number, we immediately get:

\begin{thm}
Let $f:X\to D_n(Y)$ be an $n$-valued map and $g:X\to Y$ a map.
Let $\hat{X}$ be the splitting space of $f$, with covering map $\hat{p}:\hat{X}\to X$, and write $f\circ \hat{p}=\{\hat{f}_1,\ldots,\hat{f}_n\}$ and $g\circ \hat{p}=\hat{g}$. Then
\[
L(f,g)=\frac{1}{[\Pi:S]}\sum_{i=1}^n L(\hat{f}_i,\hat{g}).
\]
\end{thm}

\begin{proof}
We have \begin{align*}
L(f,g) &= \ind(f,g,X) \\
&= \frac{1}{[\Pi:S]}\sum_{i=1}^n \ind_1(\hat{f}_i,\hat{g},\hat{X}) \\
&= \frac{1}{[\Pi:S]}\sum_{i=1}^n L(\hat{f}_i,\hat{g}). \qedhere
\end{align*}
\end{proof}

\begin{rmk}\label{rmk:av-L}
The Lefschetz numbers of the single-valued maps $\hat{f}_i$ and $\hat{g}$ can be computed using the classical definition of the single-valued Lefschetz number: let $(\hat{f}_i)_q:H_q(\hat{X},\QQ)\to H_q(Y,\QQ)$ be the morphisms on homology induced by $\hat{f}_i$, and $\hat{g}^q:H^q(Y,\QQ)\to H^q(\hat{X},\QQ)$ the morphisms on cohomology induced by $\hat{g}$. Let $D_{\hat{X}}:H_q(\hat{X})\to H^{n-q}(\hat{X})$ and $D_Y:H_q(Y)\to H^{n-q}(Y)$ denote the Poincaré duality isomorphisms. Then \[
L(\hat{f}_i,\hat{g})=\sum_q (-1)^q \,\text{tr}(D_{\hat{X}}^{-1}\circ \hat{g}^{n-q}\circ D_Y\circ (\hat{f}_i)_q).
\]
Thus, the Lefschetz number $L(f,g)$ can equivalently be computed as \[
L(f,g)=\frac{1}{[\Pi:S]}\sum_{i=1}^n \sum_q (-1)^q \,\text{tr}(D_{\hat{X}}^{-1}\circ \hat{g}^{n-q}\circ D_Y\circ (\hat{f}_i)_q).
\]

In the setting of fixed points of $n$-valued maps, the Lefschetz number was defined in a way generalising the classical single-valued definition, by Brown \cite{brown2007}. Given an $n$-valued map $f:X\to D_n(X)$ on a finite polyhedron $X$, he defines \[
L(f)=\sum_{q} (-1)^q \,\text{tr}(f_{*q})
\] 
where the morphisms $f_{*q}:H_q(X,\QQ)\to H_q(X,\QQ)$ are defined as follows. Let $f':X=|K|\to Y=|L|$ be a simplicial map homotopic to $f$, where simplicial means: for any simplex $\sigma\in K$, the restriction of $f'$ to $|\bar{\sigma}|$ (the geometric realisation of the closure of $\sigma$) splits as $\{f'_1,\ldots,f'_n\}$, where each map $f'_i:|\bar{\sigma}|\to |L|$ is simplicial. One obtains a chain map defined on $k$-simplices by \[
C_k(K)\to C_k(L):\sigma\mapsto \tau_1+\ldots+\tau_n
\]
with $\bar{\tau}_i=f'_i(\bar{\sigma})$ for $i\in\{1,\ldots,n\}$; its induced homology morphism is the map $f_{*}:H_*(X,\QQ)\to H_*(X,\QQ)$.

We can now show that our Lefschetz number (which we originally defined as $L(f,g)=\ind(f,g,X)$) satisfies \begin{equation}\label{eq:L-tr}
L(f,g)=\sum_{q} (-1)^q \,\text{tr}(D_X^{-1}\circ g^{n-q}\circ D_Y\circ f_{*q}).
\end{equation}
Thus, in the special case of fixed points of $n$-valued maps, our Lefschetz number coincides with Brown's. (Also, it follows that Brown's Lefschetz number satisfies $L(f)=\ind(f,X)$, which was not shown in \cite{brown2007}.)

To show \eqref{eq:L-tr}, we observe that for each $q$, \begin{equation}\label{eq:L-tr-help}
\sum_{i=1}^n (\hat{f}_i)_q = f_{*q}\circ \hat{p}_q.
\end{equation}
Indeed, let $f'$ and $\hat{p}'$ be simplicial maps homotopic to $f$ and $\hat{p}$, respectively. Then $f'\circ \hat{p}'=\vcentcolon \hat{f}'$ is a simplicial map homotopic to $\hat{f}$, which also splits as $\{\hat{f}'_1,\ldots,\hat{f}'_n\}$ with (up to reordering) $\hat{f}'_i$ homotopic to $\hat{f}_i$ for each $i$. The morphism $\hat{p}_q:H_q(\hat{X},\QQ)\to H_q(X,\QQ)$ is induced by the chain map sending a simplex $\hat{\sigma}$ in $\hat{X}$ to the simplex $\sigma$ in $X$ such that $\hat{p}'(\bar{\hat{\sigma}})=\bar{\sigma}$. On $\bar{\sigma}$, the map $f'$ splits as $\{f'_1,\ldots,f'_n\}$, and the chain map that induces $f_{*q}$ sends $\sigma$ to $\tau_1+\ldots+\tau_n$ with $f'_i(\bar{\sigma})=\bar{\tau}_i$ for all $i$. Since $f'\circ \hat{p}'=\vcentcolon \hat{f}'$, we may reorder the factors $f'_i$ such that $f'_i\circ \hat{p}'=\hat{f}'_i$ for all $i$, hence $\hat{f}'_i(\bar{\hat{\sigma}})=\bar{\tau}_i$ for all $i$. Then the chain map that induces $(\hat{f}_i)_q$ sends $\hat{\sigma}$ to $\tau_i$, for all $i$. On the other hand, by the above, the chain map that induces $f_{*q}\circ \hat{p}_q$ sends $\hat{\sigma}$ to $\tau_1+\ldots+\tau_n$. So \eqref{eq:L-tr-help} holds for the chain maps defining the homology morphisms, hence it must also hold for the morphisms themselves.

On the other hand, the composition \[
\hat{p}_q\circ D_{\hat{X}}^{-1}\circ \hat{p}^{n-q}\circ D_X
\]
is given by multiplication with the degree of $\hat{p}$, i.e.\ $[\Pi:S]$ (see e.g.\ \cite[Proposition 2.6(4)]{delaharpe}), so we have 
\begin{align*}
L(f,g)&=\frac{1}{[\Pi:S]}\sum_{i=1}^n \sum_q (-1)^q \,\text{tr}(D_{\hat{X}}^{-1}\circ \hat{g}^{n-q}\circ D_Y\circ (\hat{f}_i)_q) \\
&=\frac{1}{[\Pi:S]} \sum_q (-1)^q \,\text{tr}(D_{\hat{X}}^{-1}\circ \hat{g}^{n-q}\circ D_Y\circ \sum_{i=1}^n(\hat{f}_i)_q) \\
&=\frac{1}{[\Pi:S]} \sum_q (-1)^q \,\text{tr}(D_{\hat{X}}^{-1}\circ (g\circ \hat{p})^{n-q}\circ D_Y\circ f_{*q}\circ \hat{p}_q) \\
&=\frac{1}{[\Pi:S]} \sum_q (-1)^q \,\text{tr}(\hat{p}_q\circ D_{\hat{X}}^{-1}\circ \hat{p}^{n-q} \circ g^{n-q}\circ D_Y\circ f_{*q}) \\
&=\frac{1}{[\Pi:S]} \sum_q (-1)^q \,\text{tr}([\Pi:S] D_X^{-1}\circ g^{n-q}\circ D_Y\circ f_{*q}) \\
&=\sum_q (-1)^q \,\text{tr}(D_X^{-1}\circ g^{n-q}\circ D_Y\circ f_{*q}).
\end{align*}
\end{rmk}

\subsection{Comparing coincidence classes of $(f,g)$ and $(\hat{f}_i,\hat{g})$}\label{subsec:comparing}

From now on, fix a lift $\tilde{f}=(\tilde{f}_1,\ldots,\tilde{f}_n):\tilde{X}\to F_n(\tilde{Y},\Delta)$ for $f$ with induced morphism $\varphi=(\varphi_1,\ldots,\varphi_n;\sigma):\Pi\to \Delta^n\rtimes \S_n$, and a lift $\tilde{g}:\tilde{X}\to \tilde{Y}$ for $g$ with induced morphism $\psi:\Pi\to \Delta$. 

We may order the factors of $\hat{f}=\{\hat{f}_1,\ldots,\hat{f}_n\}$ so that $\tilde{f}_i$ is a lift of $\hat{f}_i$, for each $i$. The associated morphism of covering groups is $\varphi'_i:S\to \Delta$, the restriction of $\varphi_i$: \[
\forall \gamma\in S: \; \tilde{f}_i\gamma=\varphi_i(\gamma)\tilde{f}_{\sigma_\gamma^{-1}(i)}=\varphi_i(\gamma)\tilde{f}_i.
\]
Similarly, $\tilde{g}$ is a lift of $\hat{g}$ with morphism $\psi':S\to \Delta$, the restriction of $\psi$. The decomposition of $\Coin(\hat{f}_i,\hat{g})$ into coincidence classes is \[
\Coin(\hat{f}_i,\hat{g})=\bigsqcup_{[\alpha]\in \R[\varphi'_i,\psi']}\check{p}\Coin(\alpha\tilde{f}_i,\tilde{g}),
\]
where $\R[\varphi'_i,\psi']$ is the set of equivalence classes of $\Delta$ for \[
\alpha\sim_{\varphi'_i,\psi'} \beta \iff \exists \gamma\in S: \alpha=\psi(\gamma)\beta\varphi_i(\gamma^{-1}).
\]

On the other hand, recall that \[
\Coin(f,g)=\bigsqcup_{[(\alpha,i)]\in \R[\varphi,\psi]} p\Coin(\alpha\tilde{f}_i,\tilde{g})
\]
where \[
(\alpha,i)\sim_{\varphi,\psi} (\beta,j) \iff \exists \gamma\in \Pi:i=\sigma_\gamma(j),\; \alpha=\psi(\gamma)\beta\varphi_j(\gamma^{-1}).
\]
To further decompose this union, note that the morphism $\sigma:\Pi\to \S_n$ defines an equivalence relation $\sim$ on $\{1,\ldots,n\}$ by \[
i\sim j \iff \exists \gamma\in \Pi:i=\sigma_\gamma(j).
\]
We call the equivalence classes for this relation the \emph{$\sigma$-classes}, and write the $\sigma$-class of $i$ as $[i]_\sigma$. For each $i\in \{1,\ldots,n\}$, let \[
S_i=\{\gamma\in \Pi \mid \sigma_\gamma(i)=i \}
\]
be the stabiliser of $i$, which is a finite index subgroup of $\Pi$.

If $i=\sigma_\gamma(j)$, then for all $\alpha\in \Delta$, we have $(\alpha,i)\sim_{\varphi,\psi}(\alpha^\gamma,j)$ with \[
\alpha^\gamma=\psi(\gamma)^{-1}\alpha\varphi_i(\gamma).
\]
Indeed, \begin{align*}
\alpha&=\psi(\gamma)\alpha^\gamma\varphi_i(\gamma)^{-1}\\&=\psi(\gamma)\alpha^\gamma\varphi_{\sigma_\gamma^{-1}(i)}(\gamma^{-1})\\&=\psi(\gamma)\alpha^\gamma\varphi_j(\gamma^{-1}).
\end{align*}
In particular, if $i\sim j$, then \[
\{[(\alpha,i)]_{\varphi,\psi} \mid \alpha\in \Delta \}=\{[(\alpha,j)]_{\varphi,\psi} \mid \alpha\in \Delta \}.
\]
On the other hand, for a fixed $i\in \{1,\ldots,n\}$, we have \begin{align*}
(\alpha,i)\sim_{\varphi,\psi} (\beta,i) &\iff \exists \gamma\in \Pi:\sigma_\gamma(i)=i,\;\alpha=\psi(\gamma)\beta\varphi_i(\gamma^{-1}) \\
&\iff \exists \gamma\in S_i:\alpha=\psi(\gamma)\beta\varphi_i(\gamma^{-1}).
\end{align*}
We will (slightly abusively) write this as $\alpha\sim_{\varphi_i,\psi}\beta$, and $\Delta/{\sim}_{\varphi_i,\psi}=\R[\varphi_i,\psi]$.
Together, we can write \[
\Coin(f,g)=\bigsqcup_{[i]_\sigma} \bigsqcup_{[\alpha]\in \R[\varphi_i,\psi]} p\Coin(\alpha\tilde{f}_i,\tilde{g}).
\]

If $i=\sigma_\gamma(j)$, note that \begin{align*}
\alpha\sim_{\varphi_i,\psi}\beta &\iff (\alpha,i)\sim_{\varphi,\psi}(\beta,i) \\
&\iff (\alpha^\gamma,j)\sim_{\varphi,\psi}(\beta^\gamma,j) \\
&\iff \alpha^\gamma\sim_{\varphi_j,\psi}\beta^\gamma.
\end{align*}
Thus, the assignment $\alpha\mapsto \alpha^\gamma$ defines a one-to-one correspondence between $\R[\varphi_i,\psi]$ and $\R[\varphi_j,\psi]$. 

The following lemma gives an explicit decomposition of the lifts of coincidence classes of $(f,g)$ in terms of coincidence classes of $(\hat{f}_i,\hat{g})$.
We will not explicitly need this to prove our averaging formulas, but it will serve an important purpose in section \ref{sec:roots}, where we consider the special case of roots.

\begin{lem}\label{lem:coin-ij}
If $i\not\sim j$, then \[
\Coin(\hat{f}_j,\hat{g})\cap \hat{p}^{-1}p\Coin(\alpha\tilde{f}_i,\tilde{g})=\varnothing
\] 
for all $\alpha\in \Delta$. If $i\sim j$, say $i=\sigma_\gamma(j)$, then \[
\Coin(\hat{f}_j,\hat{g})\cap \hat{p}^{-1}p\Coin(\alpha\tilde{f}_i,\tilde{g})=\bigsqcup_{\substack{[\beta]\in \R[\varphi'_j,\psi'] \\ \beta\sim_{\varphi_j,\psi}\,\alpha^\gamma}} \check{p}\Coin(\beta\tilde{f}_j,\tilde{g})
\]
for all $\alpha\in \Delta$.
\end{lem}

\begin{proof}
For all $j\in \{1,\ldots,n\}$, \[
\hat{p}\Coin(\hat{f}_j,\hat{g})=\hat{p}\left(\bigcup_{\beta\in \Delta}\check{p}\Coin(\beta\tilde{f}_j,\tilde{g})\right)=\bigcup_{\beta\in \Delta}p\Coin(\beta\tilde{f}_j,\tilde{g}).
\]
If $i\not\sim j$, we know $(\alpha,i)\not\sim (\beta,j)$ for any $\alpha,\beta\in \Delta$, so \[
p\Coin(\alpha\tilde{f}_i,\tilde{g})\cap p\Coin(\beta\tilde{f}_j,\tilde{g})=\varnothing
\]
for all $\beta\in \Delta$, and therefore \[
\hat{p}\Coin(\hat{f}_j,\hat{g})\cap p\Coin(\alpha\tilde{f}_i,\tilde{g})=\varnothing.
\]
It follows that \begin{align*}
\Coin(\hat{f}_j,\hat{g})\cap \hat{p}^{-1}p\Coin(\alpha\tilde{f}_i,\tilde{g})\subseteq \hat{p}^{-1}(\hat{p}\Coin(\hat{f}_j,\hat{g})\cap p\Coin(\alpha\tilde{f}_i,\tilde{g}))=\varnothing.
\end{align*}

Now suppose $i\sim j$; say $i=\sigma_\gamma(j)$. Then $(\alpha,i)\sim_{\varphi,\psi} (\alpha^\gamma,j)$, so \[
p\Coin(\alpha\tilde{f}_i,\tilde{g})=p\Coin(\alpha^\gamma\tilde{f}_j,\tilde{g}).
\]
Since \[
\Coin(\hat{f}_j,\hat{g})=\bigsqcup_{[\beta]\in \R[\varphi'_j,\psi']}\check{p}\Coin(\beta\tilde{f}_j,\tilde{g}),
\]
it suffices to show that \[
\check{p}\Coin(\beta\tilde{f}_j,\tilde{g})\cap \hat{p}^{-1}p\Coin(\alpha^\gamma\tilde{f}_j,\tilde{g})=\varnothing
\]
if $\beta\not\sim_{\varphi_j,\psi}\alpha^\gamma$, and \[
\check{p}\Coin(\beta\tilde{f}_j,\tilde{g})\subseteq \hat{p}^{-1}p\Coin(\alpha^\gamma\tilde{f}_j,\tilde{g})
\]
if $\beta\sim_{\varphi_j,\psi}\alpha^\gamma$.

For the first statement, we have \[
p\Coin(\beta\tilde{f}_j,\tilde{g})\cap p\Coin(\alpha^\gamma\tilde{f}_j,\tilde{g})=\varnothing
\]
if $\beta\not\sim_{\varphi_j,\psi}\alpha^\gamma$, so certainly \begin{align*}
\check{p}\Coin(\beta\tilde{f}_j,\tilde{g})\cap \hat{p}^{-1}p\Coin(\alpha^\gamma\tilde{f}_j,\tilde{g})\subseteq \hat{p}^{-1}(p\Coin(\beta\tilde{f}_j,\tilde{g})\cap p\Coin(\alpha^\gamma\tilde{f}_j,\tilde{g}))=\varnothing.
\end{align*}

For the second statement, suppose $\beta\sim_{\varphi_j,\psi}\alpha^\gamma$. If $\hat{x}\in \check{p}\Coin(\beta\tilde{f}_j,\tilde{g})$, take $\tilde{x}\in \Coin(\beta\tilde{f}_j,\tilde{g})$ with $\check{p}(\tilde{x})=\hat{x}$. Then \[
\hat{p}(\hat{x})=p(\tilde{x})\in p\Coin(\beta\tilde{f}_j,\tilde{g})=p\Coin(\alpha^\gamma\tilde{f}_j,\tilde{g}),
\]
so $\hat{x}\in \hat{p}^{-1}p\Coin(\alpha^\gamma\tilde{f}_j,\tilde{g})$.
\end{proof}

In particular, we can write the index of a coincidence class of $(f,g)$ in terms of indices of coincidence classes of $(\hat{f}_i,\hat{g})$: for every $[(\alpha,i)]\in \R[\varphi,\psi]$, we have
\begin{align}\label{eq:ind-cc}
\ind(f,g,p\Coin(\alpha\tilde{f}_i,\tilde{g}))&=\ind(f,g,U) \notag \\ 
&=\frac{1}{[\Pi:S]}\sum_{j=1}^{n}\ind_1(\hat{f}_j,\hat{g},\hat{p}^{-1}(U)).
\end{align}
where $U\subseteq X$ is an open such that $p\Coin(\alpha\tilde{f}_i,\tilde{g})=\Coin(f,g)\cap U$.
Since \[
\Coin(\hat{f}_i,\hat{g})\cap\hat{p}^{-1}(U)=\Coin(\hat{f}_i,\hat{g})\cap \hat{p}^{-1}p\Coin(\alpha\tilde{f}_i,\tilde{g}),
\]
it follows from excision and Lemma \ref{lem:coin-ij} that \eqref{eq:ind-cc} can be rewritten as \[
\frac{1}{[\Pi:S]}\sum_{j\sim i}\sum_{\substack{[\beta]\in \R[\varphi'_j,\psi'] \\ \beta\sim_{\varphi_j,\psi}\,\alpha^\gamma \\ \text{(with $i=\sigma_\gamma(j)$)}}}\ind_1(\hat{f}_j,\hat{g},\check{p}\Coin(\beta\tilde{f}_j,\tilde{g})).
\]

\subsection{Averaging formula for the Reidemeister number} From the results of the previous section, it follows that \[
R(\hat{f}_i,\hat{g})=R(\varphi'_i,\psi')
\]
for all $i$, and \[
R(f,g)=R(\varphi,\psi)=\sum_{[i]_\sigma} R(\varphi_i,\psi).
\]
Since $R(\varphi_i,\psi)=R(\varphi_j,\psi)$ if $i\sim j$, we can also write \[
R(f,g)=\sum_{i=1}^{n}\frac{1}{[\Pi:S_i]} \,R(\varphi_i,\psi)
\]
where $[\Pi:S_i]$ is the number of elements of $[i]_\sigma$, by the orbit-stabiliser theorem.

\begin{lem}\label{lem:R-lift}
Let $r:\R[\varphi'_i,\psi']\to \R[\varphi_i,\psi]$ denote the natural projection. For all $\alpha\in \Delta$, with corresponding Reidemeister class $[\alpha]_{\varphi_i,\psi}\in \R[\varphi_i,\psi]$, \[
|r^{-1}([\alpha]_{\varphi_i,\psi})|=\frac{[S_i:S]}{|u_i(\coin(\tau_\alpha\varphi_i,\psi))|}
\]
where $u_i:S_i\to S_i/S$ is the projection, and  \[
\coin(\tau_\alpha\varphi_i,\psi)=\{\gamma\in S_i \mid \alpha\varphi_i(\gamma)\alpha^{-1}=\psi(\gamma) \}
\]
is the coincidence set of the morphisms $\tau_\alpha\varphi_i,\psi:S_i\to \Delta$.
\end{lem}

\begin{proof}	
For $\alpha\in \Delta$, consider the action of $S_i/S$ on $r^{-1}([\alpha]_{\varphi_i,\psi})$ given by \[
\bar{\gamma}\cdot [\beta]_{\varphi'_i,\psi'} = [\psi(\gamma)\beta\varphi_i(\gamma^{-1})]_{\varphi'_i,\psi'}
\]
(where $[\,\cdot\,]_{\varphi'_i,\psi'}$ denote the classes in $\R[\varphi'_i,\psi']$).
It is easily verified that this action is well-defined and transitive, and that the stabilizer of $[\alpha]_{\varphi'_i,\psi'}$ is $u_i(\coin(\tau_\alpha\varphi_i,\psi))$. The result follows from the orbit-stabilizer theorem.
\end{proof}

We can now show:

\begin{thm}
Let $f:X\to D_n(Y)$ be an $n$-valued map and $g:X\to Y$ a map.
Let $\hat{X}$ be the splitting space of $f$, with covering map $\hat{p}:\hat{X}\to X$, and write $f\circ \hat{p}=\{\hat{f}_1,\ldots,\hat{f}_n\}$ and $g\circ \hat{p}=\hat{g}$. Then
\[
R(f,g)\geq\frac{1}{[\Pi:S]}\sum_{i=1}^n R(\hat{f}_i,\hat{g})
\]
where the equality holds if and only if either $R(\hat{f}_i,\hat{g})=\infty$ for some $i$, or $\coin(\tau_\alpha\varphi_i,\psi)\subseteq S$ for all $\alpha$ and $i$.
\end{thm}

\begin{proof}
By the above, we get \begin{align*}
R(f,g)&=\sum_{i=1}^n \frac{1}{[\Pi:S_i]}R(\varphi_i,\psi) \\
&=\sum_{i=1}^n \frac{1}{[\Pi:S_i]}\sum_{[\alpha]\in \R[\varphi'_i,\psi']} \frac{|u_i(\coin(\tau_\alpha\varphi_i,\psi))|}{[S_i:S]} \\
&=\frac{1}{[\Pi:S]}\sum_{i=1}^{n}\sum_{[\alpha]\in \R[\varphi'_i,\psi']}|u_i(\coin(\tau_\alpha\varphi_i,\psi))|.
\end{align*}
Since $1\in \coin(\tau_\alpha\varphi_i,\psi)$, we have $|u_i(\coin(\tau_\alpha\varphi_i,\psi))|\geq 1$ for all $\alpha$ and $i$.
If $R(\hat{f}_i,\hat{g})=\infty$, then \[
\sum_{[\alpha]\in \R[\varphi'_i,\psi']}|u_i(\coin(\tau_\alpha\varphi_i,\psi))|=\infty =R(\hat{f}_i,\hat{g}).
\]
If $R(\hat{f}_i,\hat{g})<\infty$, then \[
\sum_{[\alpha]\in \R[\varphi'_i,\psi']}|u_i(\coin(\tau_\alpha\varphi_i,\psi))|\geq \sum_{[\alpha]\in \R[\varphi'_i,\psi']} 1 =R(\hat{f}_i,\hat{g}),
\]
with an equality if and only if $\coin(\tau_\alpha\varphi_i,\psi)\subseteq S$ for all $\alpha$.
\end{proof}

\subsection{Averaging formula for the Nielsen number}

To compare essentiality of coincidence classes of $(f,g)$ and $(\hat{f}_i,\hat{g})$, we need a more explicit relation between these indices than the one we derived at the end of section \ref{subsec:comparing}. To derive this explicit relationship (Lemma \ref{lem:ind-lift}), we use the equivalent definition of the index from Remark \ref{rmk:ind-equiv}, for which we must first deform $g$ to a map $g'$ such that $\Coin(f,g')$ is finite. Since $N(f,g)=N(f,g')$, we will for simplicity assume that $g'=g$, i.e.\ that $\Coin(f,g)$ is finite, for the remainder of this section.

\begin{lem}
For all $x\in p\Coin(\alpha\tilde{f}_i,\tilde{g})$, we have \[
|\hat{p}^{-1}(x)\cap \check{p}\Coin(\alpha\tilde{f}_i,\tilde{g})|=|u_i(\coin(\tau_\alpha\varphi_i,\psi))|.
\]
\end{lem}

\begin{proof}
Fix $\hat{x}\in \hat{p}^{-1}(x)\cap \check{p}\Coin(\alpha\tilde{f}_i,\tilde{g})$. We will express all other elements of $\hat{p}^{-1}(x)\cap \check{p}\Coin(\alpha\tilde{f}_i,\tilde{g})$ in terms of $\hat{x}$.

Write $\hat{x}=\check{p}(\tilde{x})$ with $\tilde{x}\in \Coin(\alpha\tilde{f}_i,\tilde{g})$. For any other element $\hat{x}'\in \hat{p}^{-1}(x)\cap \check{p}\Coin(\alpha\tilde{f}_i,\tilde{g})$, write $\hat{x}'=\check{p}(\tilde{x}')$ with $\tilde{x}'\in \Coin(\alpha\tilde{f}_i,\tilde{g})$. Since \[
p(\tilde{x})=\hat{p}(\hat{x})=x=\hat{p}(\hat{x}')=p(\tilde{x}'),
\]
there is a $\gamma\in \Pi$ so that $\tilde{x}'=\gamma\tilde{x}$. Since $\tilde{x}'\in \Coin(\alpha\tilde{f}_i,\tilde{g})$, we have \[
\alpha\tilde{f}_i(\gamma\tilde{x})=\tilde{g}(\gamma\tilde{x}) \iff \alpha\varphi_i(\gamma)\tilde{f}_{\sigma_\gamma^{-1}(i)}(\tilde{x})=\psi(\gamma)\tilde{g}(\tilde{x}).
\]
Using that $\tilde{x}\in \Coin(\alpha\tilde{f}_i,\tilde{g})$, we get \begin{equation}\label{eq:lem:N}
\alpha\varphi_i(\gamma)\tilde{f}_{\sigma_\gamma^{-1}(i)}(\tilde{x})=\psi(\gamma)\alpha\tilde{f}_i(\tilde{x}).
\end{equation}
It follows that $\sigma_\gamma^{-1}(i)=i$ and $\alpha\varphi_i(\gamma)=\psi(\gamma)\alpha$. That is, $\gamma\in S_i$ and $\alpha\varphi_i(\gamma)\alpha^{-1}=\psi(\gamma)$, so $\gamma\in \coin(\tau_\alpha\varphi_i,\psi)$.

We have $\hat{x}'=\check{p}(\gamma\tilde{x})=u_i(\gamma)\check{p}(\tilde{x})=u_i(\gamma)\hat{x}$. This shows that \[
\hat{p}^{-1}(x)\cap \check{p}\Coin(\alpha\tilde{f}_i,\tilde{g})\subseteq\{u_i(\gamma)\hat{x} \mid \gamma\in \coin(\tau_\alpha\varphi_i,\psi) \}.
\]
Note that the converse also holds: if $\gamma\in \coin(\tau_\alpha\varphi_i,\psi)$, it follows from \eqref{eq:lem:N} that $\gamma\tilde{x}\in \Coin(\alpha\tilde{f}_i,\tilde{g})$, so $u_i(\gamma)\hat{x}=\check{p}(\gamma\tilde{x})\in \hat{p}^{-1}(x)\cap \check{p}\Coin(\alpha\tilde{f}_i,\tilde{g})$. Hence we have \[
\hat{p}^{-1}(x)\cap \check{p}\Coin(\alpha\tilde{f}_i,\tilde{g})=\{u_i(\gamma)\hat{x} \mid \gamma\in \coin(\tau_\alpha\varphi_i,\psi) \}.
\]
The number of elements in this set is $|u_i(\coin(\tau_\alpha\varphi_i,\psi))|$.
\end{proof}

In other words, \[
\hat{p}:\check{p}\Coin(\alpha\tilde{f}_i,\tilde{g}) \to p\Coin(\alpha\tilde{f}_i,\tilde{g})
\]
is a $|u_i(\coin(\tau_\alpha\varphi_i,\psi))|$-fold covering map.

\begin{lem}\label{lem:ind-lift}
For all $\alpha\in \Delta$ and $i\in\{1,\ldots,n\}$, \[
\ind_1(\hat{f}_i,\hat{g},\check{p}\Coin(\alpha\tilde{f}_i,\tilde{g}))=|u_i(\coin(\tau_\alpha\varphi_i,\psi))|\ind(f,g,p\Coin(\alpha\tilde{f}_i,\tilde{g})).
\]
\end{lem}

\begin{proof}
As assumed at the beginning of this section, $\Coin(f,g)$ is finite; in particular so is $p\Coin(\alpha\tilde{f}_i,\tilde{g})$. As we saw in Remark \ref{rmk:ind-equiv}, we can write \[
\ind(f,g,p\Coin(\alpha\tilde{f}_i,\tilde{g}))=\sum_{x\in p\Coin(\alpha\tilde{f}_i,\tilde{g})}\ind_1(f^x_{i_x},g,x)
\]
where for all $x\in p\Coin(\alpha\tilde{f}_i,\tilde{g})$, $\{f^x_1,\ldots,f^x_n\}$ is a splitting of $f$ on a neighborhood of $x$, and $i_x\in \{1,\ldots,n\}$ is the unique index such that $f^x_{i_x}(x)=g(x)$.

On this neighborhood of $x$, we have \[
\{\hat{f}_1,\ldots,\hat{f}_n\}=f\circ \hat{p}=\{f^x_1\circ \hat{p},\ldots,f^x_n\circ \hat{p}\},
\]
so we may order the factors $f^x_1,\ldots,f^x_n$ so that $f^x_j\circ \hat{p}=\hat{f}_j$ for all $j$. If we write $x=p(\tilde{x})$ with $\alpha\tilde{f}_i(\tilde{x})=\tilde{g}(\tilde{x})$, then \begin{align*}
q(\alpha\tilde{f}_i(\tilde{x}))=q(\tilde{g}(\tilde{x})) &\iff \hat{f}_i(\check{p}(\tilde{x}))=g(p(\tilde{x})) \\
&\iff f^x_i(\hat{p}(\check{p}(\tilde{x}))) = g(x) \\
&\iff f^x_i(x)=g(x).
\end{align*}
Thus, $i_x=i$ for all $x$.

Since the covering map $\hat{p}$ is an orientation-preserving local homeomorphism, and $\ind_1$ is a local invariant, we have \[
\ind_1(f^x_i,g,x)=\ind_1(f^x_i\circ \hat{p},g\circ \hat{p},\hat{x})=\ind_1(\hat{f}_i,\hat{g},\hat{x})
\]
for all $\hat{x}\in \hat{p}^{-1}(x)$. We can now write \begin{align*}
\ind_1(\hat{f}_i,\hat{g},\check{p}\Coin(\alpha\tilde{f}_i,\tilde{g})) &= \sum_{\hat{x}\in \check{p}\Coin(\alpha\tilde{f}_i,\tilde{g})} \ind_1(\hat{f}_i,\hat{g},\hat{x}) \\
&= \sum_{\hat{x}\in \check{p}\Coin(\alpha\tilde{f}_i,\tilde{g})} \ind_1(f^x_i,g,\hat{p}(\hat{x})) \\
&= \sum_{x\in p\Coin(\alpha\tilde{f}_i,\tilde{g})}\sum_{\hat{x}\in \hat{p}^{-1}(x)\cap \check{p}\Coin(\alpha\tilde{f}_i,\tilde{g})} \ind_1(f^x_i,g,x) \\
&= \sum_{x\in p\Coin(\alpha\tilde{f}_i,\tilde{g})} |u_i(\coin(\tau_\alpha\varphi_i,\psi))|\, \ind_1(f^x_i,g,x) \\
&= |u_i(\coin(\tau_\alpha\varphi_i,\psi))|\,\ind(f,g,p\Coin(\alpha\tilde{f}_i,\tilde{g})). \qedhere
\end{align*}
\end{proof}

\begin{thm}
Let $f:X\to D_n(Y)$ be an $n$-valued map and $g:X\to Y$ a map.
Let $\hat{X}$ be the splitting space of $f$, with covering map $\hat{p}:\hat{X}\to X$, and write $f\circ \hat{p}=\{\hat{f}_1,\ldots,\hat{f}_n\}$ and $g\circ \hat{p}=\hat{g}$. Then
\[
N(f,g)\geq\frac{1}{[\Pi:S]}\sum_{i=1}^n N(\hat{f}_i,\hat{g})
\]
where the equality holds if and only if $\coin(\tau_\alpha\varphi_i,\psi)\subseteq S$ for all $\alpha$ and $i$ for which $\check{p}\Coin(\alpha\tilde{f}_i,\tilde{g})$ is an essential coincidence class of $(\hat{f}_i,\hat{g})$.
\end{thm}

\begin{proof}
Write \[
\eps(f,g,p\Coin(\alpha\tilde{f}_i,\tilde{g}))=\begin{cases}
1 & \text{if }\ind(f,g,p\Coin(\alpha\tilde{f}_i,\tilde{g}))\neq 0 \\
0 & \text{otherwise}
\end{cases}
\]
and \[
\eps_1(\hat{f}_i,\hat{g},\check{p}\Coin(\alpha\tilde{f}_i,\tilde{g}))=\begin{cases}
1 & \text{if }\ind_1(\hat{f}_i,\hat{g},\check{p}\Coin(\alpha\tilde{f}_i,\tilde{g}))\neq 0 \\
0 & \text{otherwise}.
\end{cases}
\]
Then \[
N(\hat{f}_i,\hat{g})=\sum_{[\alpha]\in \R[\varphi'_i,\psi']} \eps_1(\hat{f}_i,\hat{g},\check{p}\Coin(\alpha\tilde{f}_i,\tilde{g}))
\]
for all $i$, and
\[
N(f,g)=\sum_{[i]_\sigma} \sum_{[\alpha]\in\R[\varphi_i,\psi]} \eps(f,g,p\Coin(\alpha\tilde{f}_i,\tilde{g})).
\]
If $i=\sigma_\gamma(j)$, then \begin{align*}
\sum_{[\beta]\in\R[\varphi_j,\psi]} \eps(f,g,p\Coin(\beta\tilde{f}_j,\tilde{g}))&=\sum_{[\alpha]\in\R[\varphi_i,\psi]} \eps(f,g,p\Coin(\alpha^\gamma\tilde{f}_j,\tilde{g})) \\
&=\sum_{[\alpha]\in\R[\varphi_i,\psi]} \eps(f,g,p\Coin(\alpha\tilde{f}_i,\tilde{g})),
\end{align*}
where the first equality follows from the observation above Lemma \ref{lem:coin-ij}, and the second from Propositions \ref{prop:cc-partition} and \ref{prop:cc-equiv-R}. Hence the summands are independent of the chosen representatives $i$ for the $\sigma$-classes, and we can also write \[
N(f,g)=\sum_{i=1}^n \frac{1}{[\Pi:S_i]} \sum_{[\alpha]\in\R[\varphi_i,\psi]} \eps(f,g,p\Coin(\alpha\tilde{f}_i,\tilde{g})).
\]
By Lemma \ref{lem:ind-lift}, we know \[
\eps(f,g,p\Coin(\alpha\tilde{f}_i,\tilde{g}))=\eps_1(\hat{f}_i,\hat{g},\check{p}\Coin(\alpha\tilde{f}_i,\tilde{g})).
\]
By Lemma \ref{lem:R-lift}, we can rewrite \begin{align*}
N(f,g) &= \sum_{i=1}^n \frac{1}{[\Pi:S_i]} \sum_{[\alpha]\in\R[\varphi_i,\psi]} \eps_1(\hat{f}_i,\hat{g},\check{p}\Coin(\alpha\tilde{f}_i,\tilde{g})) \\
&= \sum_{i=1}^n \frac{1}{[\Pi:S_i]} \sum_{[\alpha]\in\R[\varphi'_i,\psi']} \frac{|u_i(\coin(\tau_\alpha\varphi_i,\psi))|}{[S_i:S]} \, \eps_1(\hat{f}_i,\hat{g},\check{p}\Coin(\alpha\tilde{f}_i,\tilde{g})) \\
&= \frac{1}{[\Pi:S]} \sum_{i=1}^n \sum_{[\alpha]\in\R[\varphi'_i,\psi']}|u_i(\coin(\tau_\alpha\varphi_i,\psi))|\,\eps_1(\hat{f}_i,\hat{g},\check{p}\Coin(\alpha\tilde{f}_i,\tilde{g})) \\
&\geq \frac{1}{[\Pi:S]} \sum_{i=1}^n \sum_{[\alpha]\in\R[\varphi'_i,\psi']}\eps_1(\hat{f}_i,\hat{g},\check{p}\Coin(\alpha\tilde{f}_i,\tilde{g})) \\
&= \frac{1}{[\Pi:S]} \sum_{i=1}^n N(\hat{f}_i,\hat{g})
\end{align*}
with an equality if and only if $|u_i(\coin(\tau_\alpha\varphi_i,\psi))|=1$ for all $\alpha$ and $i$ for which $\eps_1(\hat{f}_i,\hat{g},\check{p}\Coin(\alpha\tilde{f}_i,\tilde{g}))=1$. 
\end{proof}

\section{Formulas on infra-nilmanifolds}\label{sec:inm}

Using the known formulas for coincidences of single-valued maps of infra-nilmanifolds from \cite{HLP2012}, we can obtain explicit formulas for $N(f,g)$, $R(f,g)$ and $L(f,g)$.

\begin{rmk}\label{rmk:notation-inm}
In the theorem below, we use the following notation. Suppose $\Pi_1\orb G_1$ and $\Pi_2\orb G_2$ are orientable infra-nilmanifolds of equal dimension, finitely covered by nilmanifolds $\Gamma_1\orb G_1$ and $\Gamma_2\orb G_2$ respectively, and consider a morphism $\varphi:S\to \Pi_2$, where $S$ is a finite index subgroup of $\Pi_1$. To $\varphi$ we can associate a Lie algebra morphism $\varphi_*:\g_1\to \g_2$ between the Lie algebras of $G_1$ and $G_2$ as follows.

Let $\Gamma'_1$ be a finite index subgroup of $\Gamma_1$ that is contained in $S$, such that $\varphi(\Gamma'_1)\subseteq \Gamma_2$.\footnote{For example, we can take \[
\Gamma'_1=\langle \gamma^{[\Pi_2:\Gamma_2]} \mid \gamma\in S\cap \Gamma_1 \rangle.
\]
The quotient $(S\cap \Gamma_1)/\Gamma'_1$ is finite since it is periodic and finitely generated nilpotent. Since $S\cap \Gamma_1$ is a finite index subgroup of $\Gamma_1$, it follows that $\Gamma'_1$ is a finite index subgroup of $\Gamma_1$ as well. Also, $\varphi(\Gamma'_1)\subseteq \Gamma_2$, since $\varphi(\gamma^{[\Pi_2:\Gamma_2]})=\varphi(\gamma)^{[\Pi_2:\Gamma_2]}\in \Gamma_2$ for all $\gamma\in \Pi_1$.}
Since $\Gamma'_1$ is a finite index subgroup of $\Gamma_1$, the Mal'cev completion of $\Gamma'_1$ is also $G_1$, so the restriction $\varphi:\Gamma'_1\to \Gamma_2$ extends uniquely to a Lie group morphism $G_1\to G_2$, whose Lie algebra morphism we denote by $\varphi_*:\g_1\to \g_2$.

This morphism is independent of the choice of group $\Gamma'_1$: if $\Gamma''_1$ is another group satisfying the requirements, then so is $\Gamma'_1\cap \Gamma''_1$. In particular, the restriction $\varphi:\Gamma'_1\cap \Gamma''_1\to \Gamma_2$ extends uniquely to a Lie group morphism $G_1\to G_2$. By uniqueness, the extensions of $\varphi:\Gamma'_1\to \Gamma_2$ and $\varphi:\Gamma''_1\to \Gamma_2$ must be this same morphism.

Following \cite{HLP2012}, if $\varphi_*:\g_1\to \g_2$ is a morphism between the Lie algebras, then we use \[
\det_{\Gamma_2}^{\Gamma_1}(\varphi_*)
\]
to denote the determinant of the matrix of $\varphi_*$ taken with respect to \emph{preferred bases} for $\g_1$ and $\g_2$ corresponding to $\Gamma_1$ and $\Gamma_2$. Given a lattice $\Gamma$ in a connected simply connected nilpotent Lie group $G$, a preferred basis for the Lie algebra $\g$ consists of the images under the logarithmic map $\log:G\to \g$ of a generating set for $\Gamma$ constructed in a specific way, as explained in \cite[p.~249]{kimlee2005}.\footnote{In \cite{HLP2012}, they say a preferred basis is simply the image of any generating set of $\Gamma$, with a condition on the orientation of the resulting basis, but this is wrong, as the images of an arbitrary generating set could span a smaller space than $\g$.} It is shown in \cite{kimlee2005} that the above determinant is independent of the particular choice of preferred bases.
\end{rmk}

\begin{thm}
Let $\Pi_1\orb G_1$ and $\Pi_2\orb G_2$ be orientable infra-nilmanifolds of equal dimension, finitely covered by nilmanifolds $\Gamma_1\orb G_1$ and $\Gamma_2\orb G_2$ respectively. Let $f:\Pi_1\orb G_1\to D_n(\Pi_2\orb G_2)$ be an $n$-valued map and $g:\Pi_1\orb G_1\to \Pi_2\orb G_2$ a map. Let $\varphi_i:S_i\subseteq \Pi_1\to \Pi_2$ be the morphisms induced by some lift of $f$, and $\psi:\Pi_1\to \Pi_2$ the morphism induced by some lift of $g$.
Then \begin{align*}
L(f,g)&=\frac{1}{[\Pi_1:\Gamma_1]} \sum_{i=1}^n\sum_{\bar{\alpha}\in \Pi_2/\Gamma_2} \det_{\Gamma_2}^{\Gamma_1}(\psi_*-(\tau_\alpha\varphi_i)_*) \\
R(f,g)&=\frac{1}{[\Pi_1:\Gamma_1]} \sum_{i=1}^n\sum_{\bar{\alpha}\in \Pi_2/\Gamma_2} |\!\det_{\Gamma_2}^{\Gamma_1}(\psi_*-(\tau_\alpha\varphi_i)_*)|_\infty \\
N(f,g)&=\frac{1}{[\Pi_1:\Gamma_1]} \sum_{i=1}^n\sum_{\bar{\alpha}\in \Pi_2/\Gamma_2} |\!\det_{\Gamma_2}^{\Gamma_1}(\psi_*-(\tau_\alpha\varphi_i)_*)|
\end{align*}
where $|a|_\infty=|a|$ if $a\neq 0$ and $|a|_\infty=\infty$ if $a=0$.
\end{thm}

\begin{proof}
In \cite[Theorem 6.11]{HLP2012} it is show that, given two infra-nilmanifolds $\Pi_1\orb G_1$ and $\Pi_2\orb G_2$, finitely covered by nilmanifolds $\Gamma_1\orb G_1$ and $\Gamma_2\orb G_2$, and two single-valued maps $f,g:\Pi_1\orb G_1 \to \Pi_2\orb G_2$ with respective induced morphisms $\varphi,\psi:\Pi_1\to \Pi_2$, \begin{align*}
L(f,g)&=\frac{1}{[\Pi_1:\Gamma_1]} \sum_{\bar{\alpha}\in \Pi_2/\Gamma_2} \det_{\Gamma_2}^{\Gamma_1}(\psi_*-(\tau_\alpha\varphi)_*) \\
R(f,g)&=\frac{1}{[\Pi_1:\Gamma_1]} \sum_{\bar{\alpha}\in \Pi_2/\Gamma_2} |\!\det_{\Gamma_2}^{\Gamma_1}(\psi_*-(\tau_\alpha\varphi)_*)|_\infty \\
N(f,g)&=\frac{1}{[\Pi_1:\Gamma_1]} \sum_{\bar{\alpha}\in \Pi_2/\Gamma_2} |\!\det_{\Gamma_2}^{\Gamma_1}(\psi_*-(\tau_\alpha\varphi)_*)|.
\end{align*}
If $S\orb G_1$ is the splitting space of our $n$-valued map $f$, and $\hat{f}=\{\hat{f}_1,\ldots,\hat{f}_n\}:S\orb G_1\to D_n(\Pi_2\orb G_2)$ and $\hat{g}:S\orb G_1\to \Pi_2\orb G_2$ are the corresponding maps, then we can apply this to the pairs $\hat{f}_i,\hat{g}:S\orb G_1 \to \Pi_2\orb G_2$, whose induced morphisms we know are the restrictions $\varphi'_i:S\to \Pi_2$ and $\psi':S\to \Pi_2$. As a nilmanifold that finitely covers $S\orb G_1$, we can take $\Gamma'_1\orb G_1$ with $\Gamma'_1=\Gamma_1\cap S$. We get: \begin{align*}
L(\hat{f}_i,\hat{g})&=\frac{1}{[S:\Gamma'_1]} \sum_{\bar{\alpha}\in \Pi_2/\Gamma_2} \det_{\Gamma_2}^{\Gamma'_1}(\psi'_*-(\tau_\alpha\varphi'_i)_*) \\
R(\hat{f}_i,\hat{g})&=\frac{1}{[S:\Gamma'_1]} \sum_{\bar{\alpha}\in \Pi_2/\Gamma_2} |\!\det_{\Gamma_2}^{\Gamma'_1}(\psi'_*-(\tau_\alpha\varphi'_i)_*)|_\infty \\
N(\hat{f}_i,\hat{g})&=\frac{1}{[S:\Gamma'_1]} \sum_{\bar{\alpha}\in \Pi_2/\Gamma_2} |\!\det_{\Gamma_2}^{\Gamma'_1}(\psi'_*-(\tau_\alpha\varphi'_i)_*)|.
\end{align*}
Since $\Gamma'_1$ is a finite index subgroup of $\Gamma_1$, it follows from \cite[Corollary 6.6]{HLP2012} that \[
\det_{\Gamma_2}^{\Gamma'_1}(\varphi_*)=[\Gamma_1:\Gamma'_1]\det_{\Gamma_2}^{\Gamma_1}(\varphi_*)
\]
for any Lie algebra morphism $\varphi_*:\g_1\to\g_2$. Also, it follows from the construction in Remark \ref{rmk:notation-inm} that $\psi'_*=\psi_*$ and $(\varphi'_i)_*=(\varphi_i)_*$ for all $i$ (indeed, to compute $\psi_*$, we can use the same subgroup of $S\cap \Gamma_1$ as for $\psi'_*$, which is also a finite index subgroup of $\Pi_1\cap \Gamma_1=\Gamma_1$; for $(\varphi_i)_*$ we can use the same subgroup of $S\cap \Gamma_1$ as for $(\varphi'_i)_*$, which is also a finite index subgroup of $S_i\cap \Gamma_1$).

Combining this with the averaging formulas from the previous sections, we find: \begin{align*}
L(f,g)&=\frac{1}{[\Pi_1:S]}\sum_{i=1}^{n}\frac{1}{[S:\Gamma'_1]} \sum_{\bar{\alpha}\in \Pi_2/\Gamma_2} [\Gamma_1:\Gamma'_1]\det_{\Gamma_2}^{\Gamma_1}(\psi_*-(\tau_\alpha\varphi_i)_*) \\
&= \frac{1}{[\Pi_1:\Gamma_1]}\sum_{i=1}^{n} \sum_{\bar{\alpha}\in \Pi_2/\Gamma_2} \det_{\Gamma_2}^{\Gamma_1}(\psi_*-(\tau_\alpha\varphi_i)_*)
\end{align*}
and similarly \begin{align*}
R(f,g)&\geq \frac{1}{[\Pi_1:\Gamma_1]}\sum_{i=1}^{n} \sum_{\bar{\alpha}\in \Pi_2/\Gamma_2} |\!\det_{\Gamma_2}^{\Gamma_1}(\psi_*-(\tau_\alpha\varphi_i)_*)|_\infty \\
N(f,g)&\geq \frac{1}{[\Pi_1:\Gamma_1]}\sum_{i=1}^{n} \sum_{\bar{\alpha}\in \Pi_2/\Gamma_2} |\!\det_{\Gamma_2}^{\Gamma_1}(\psi_*-(\tau_\alpha\varphi_i)_*)|.
\end{align*}
It remains to be shown that the inequalities for $R(f,g)$ and $N(f,g)$ are equalities.

For $R(f,g)$, it follows from the formula above that $R(\hat{f}_i,\hat{g})<\infty$ if and only if $\det_{\Gamma_2}^{\Gamma'_1}(\psi'_*-(\tau_\alpha\varphi'_i)_*)\neq 0$ for all $\alpha\in \Pi_2$.

For $N(f,g)$, we recall some more results from \cite{HLP2012}. Let $\Gamma''_1$ be a finite index subgroup of $\Gamma'_1$ such that $\varphi'_i(\Gamma''_1),\psi'(\Gamma''_1)\subseteq \Gamma_2$. Then there exist lifts $\bar{f}_i,\bar{g}:\Gamma''_1\orb G_1\to \Gamma_2\orb G_2$ of $(\hat{f_i},\hat{g})$ to the nilmanifolds $\Gamma''_1\orb G_1$ and $\Gamma_2\orb G_2$. For $\bar{\alpha}\in \Pi_2/\Gamma_2$, also $(\bar{\alpha}\bar{f}_i,\bar{g})$ is a pair of lifts of $(\hat{f_i},\hat{g})$. It holds that \[
N(\bar{\alpha}\bar{f}_i,\bar{g})=|\!\det_{\Gamma_2}^{\Gamma''_1}(\psi'_*-(\tau_\alpha\varphi'_i)_*)|=[\Gamma'_1:\Gamma''_1]\,|\!\det_{\Gamma_2}^{\Gamma'_1}(\psi'_*-(\tau_\alpha\varphi'_i)_*)|,
\]
and in case the coincidence class $\check{p}\Coin(\alpha\tilde{f}_i,\tilde{g})$ of $(\hat{f}_i,\hat{g})$ is essential, then $N(\bar{\alpha}\bar{f}_i,\bar{g})\neq 0$. (To see the latter, one can show analogously to Lemma \ref{lem:ind-lift} that $\ind_1(\bar{\alpha}\bar{f}_i,\bar{g},\bar{p}\Coin(\alpha\tilde{f}_i,\tilde{g}))$ (with $\bar{p}:G_1\to \Gamma''_1\orb G_1$ the natural covering map) is a multiple of $\ind_1(\hat{f}_i,\hat{g},\check{p}\Coin(\alpha\tilde{f}_i,\tilde{g}))$, so if the second is non-zero, then so is the first; and then $N(\bar{\alpha}\bar{f}_i,\bar{g})\neq 0$ since there is a coincidence class with non-zero index.)
Thus, if the coincidence class $\check{p}\Coin(\alpha\tilde{f}_i,\tilde{g})$ is essential, then $\det_{\Gamma_2}^{\Gamma'_1}(\psi'_*-(\tau_\alpha\varphi'_i)_*)\neq 0$.

So, both for $R(f,g)$ and $N(f,g)$, it suffices to show that $\coin(\tau_\alpha\varphi_i,\psi)\subseteq S$ whenever $\det_{\Gamma_2}^{\Gamma'_1}(\psi'_*-(\tau_\alpha\varphi'_i)_*)\neq 0$. In fact, we will show that in this case $\coin(\tau_\alpha\varphi_i,\psi)=1$.

We can again take $\Gamma''_1$ as above, so that $(\tau_\alpha\varphi'_i)_*$ and $\psi'_*$ are the Lie algebra morphisms induced by the restrictions $\tau_\alpha\varphi'_i,\psi':\Gamma''_1\to \Gamma_2$. Suppose $\gamma\in \coin(\tau_\alpha\varphi_i,\psi)$. Then $\gamma^{[\Pi_1:\Gamma''_1]}\in \coin(\tau_\alpha\varphi'_i,\psi')$, since $\gamma^{[\Pi_1:\Gamma''_1]}\in S$ and \[
\tau_\alpha\varphi'_i(\gamma^{[\Pi_1:\Gamma''_1]})=(\tau_\alpha\varphi_i(\gamma))^{[\Pi_1:\Gamma''_1]}=\psi(\gamma)^{[\Pi_1:\Gamma''_1]}=\psi'(\gamma^{[\Pi_1:\Gamma''_1]}).
\]
Let $\log:G_1\to \g_1$ be the logarithmic map, and let $X=\log(\gamma^{[\Pi_1:\Gamma''_1]})$. Then \[
(\tau_\alpha\varphi'_i)_*(X)=\log(\tau_\alpha\varphi'_i(\gamma^{[\Pi_1:\Gamma''_1]}))=\log(\psi'(\gamma^{[\Pi_1:\Gamma''_1]}))=\psi'_*(X),
\]
so $(\psi'_*-(\tau_\alpha\varphi'_i)_*)(X)=0$. Since $\det_{\Gamma_2}^{\Gamma'_1}(\psi'_*-(\tau_\alpha\varphi'_i)_*)\neq 0$, it follows that $X=0$, so $\gamma^{[\Pi_1:\Gamma''_1]}=\exp(X)=1$. Since $\Pi_1$ is torsion free, it follows that $\gamma=1$.
\end{proof}

\section{Root theory}\label{sec:roots}

A Nielsen root theory for $n$-valued maps has been defined in \cite{brownkolahi}, and further studied in \cite{brown2021,browngoncalves2023}. Given an $n$-valued map $f:X\to D_n(Y)$ and a point $a\in Y$, two roots of $f$ lie in the same root class if there exists a path $c:I\to X$ between them such that if $fc=\{c_1,\ldots,c_n\}$, one of the factors $c_i$ is a contractible loop at $a$. A root class is called essential if it is geometrically essential, i.e.\ it cannot vanish after a homotopy. The Nielsen root number $\mathrm{NR}(f)$ is the number of essential root classes, and it is a homotopy-invariant lower bound for \[
\MR[f]=\{\#\Root(g) \mid g\simeq f \},
\]
where $\Root(g)=\{x\in X \mid a\in g(x) \}$.

In \cite{browngoncalves2023}, the authors define lifting classes, a Reidemeister root number, and -- if $X$ and $Y$ are closed orientable manifolds of the same dimension -- a root index, in terms of the factors $\hat{f}_i:\hat{X}\to D_n(Y)$. They call two factors $\hat{f}_i$ and $\hat{f}_j$ equivalent if, in our terminology, the indices $i$ and $j$ belong to the same $\sigma$-class, and they define the Reidemeister root number as \[
\mathrm{RR}(f)=\sum_{[i]_\sigma} \mathrm{RR}(\hat{f}_i)
\] 
where $\mathrm{RR}(\hat{f}_i)$ is the number of root classes of $\hat{f}_i$.

A result essential for the well-behavedness of this number $\mathrm{RR}(f)$, and also for the definition of the root index, is \cite[Theorem 2.4(b)]{browngoncalves2023}, which states that for every root class $R$ of $f$, for every $i$, either $\hat{p}^{-1}(R)\cap \Root(\hat{f}_i)$ is empty or it is a root class $\hat{R}_i$ of $\hat{f}_i$. However, what they actually show is that $\hat{p}^{-1}(R)$ is a union of root classes of the factors $\hat{f}_i$, which does not necessarily imply the statement; it only implies that $\hat{p}^{-1}(R)\cap \Root(\hat{f}_i)$ is a \emph{union} of root classes of $\hat{f}_i$, for each $i$.

To give a counterexample and a correction, we observe that this can be viewed as a special case of Lemma \ref{lem:coin-ij}. The roots of $f$ are the coincidences of $f$ with the constant map $g:X\to Y:x\mapsto a$ (which we will simply write as $a:X\to Y$). The corresponding map $\hat{g}:\hat{X}\to Y$ is also the constant map $a$, and for the lift $\tilde{g}:\tilde{X}\to \tilde{Y}$ we can take the constant map $\tilde{a}$ for any $\tilde{a}\in\tilde{Y}$ with $q(\tilde{a})=a$. By Theorem \ref{thm:cc-paths}, the root classes of $f$ are precisely the non-empty coincidence classes of $(f,a)$; the root classes of $\hat{f}_i$ are the non-empty coincidence classes of $(\hat{f}_i,a)$. Let us call the coincidence classes of $(f,a)$ (including the empty ones) the \emph{algebraic root classes} of $f$, and those defined using paths (which coincide with the non-empty algebraic root classes) the \emph{geometric root classes}. Thus, the algebraic root classes of $f$ are the sets $p\Coin(\alpha\tilde{f}_i,\tilde{a})=p\Root(\alpha\tilde{f}_i)$, indexed by $[(\alpha,i)]\in\R[\varphi,\psi]$, where $\varphi:\Pi\to \Delta^n\rtimes \S_n$ is the morphism of $f$ and $\psi:\Pi\to \Delta$ is the morphism of $a$; the latter is the trivial morphism. The algebraic root classes of $\hat{f}_i$ are the sets $\check{p}\Coin(\alpha\tilde{f}_i,\tilde{a})=\check{p}\Root(\alpha\tilde{f}_i)$, indexed by $[(\alpha,i)]\in\R[\varphi'_i,\psi']$, where $\varphi'_i:S\to \Delta$ is the restriction of $\varphi_i$ and $\psi':S\to \Delta$ is the restriction of $\psi$, so also the trivial morphism. Lemma \ref{lem:coin-ij} in this setting reads:

\begin{lem}
Let $R=p\Root(\alpha\tilde{f}_i)$ be an (algebraic) root class of $f$.
If $i\not\sim j$, then \[
\hat{p}^{-1}(R)\cap\Root(\hat{f}_j)=\varnothing.
\] 
If $i\sim j$, say $i=\sigma_\gamma(j)$, then \[
\hat{p}^{-1}(R)\cap\Root(\hat{f}_j)=\bigsqcup_{\substack{[\beta]\in \R[\varphi'_j,\psi'] \\ \beta\sim_{\varphi_j,\psi}\,\alpha^\gamma}} \hat{R}_{[\beta]},
\]
where $\hat{R}_{[\beta]}=\check{p}\Root(\beta\tilde{f}_j)$.
\end{lem}

To give an example where the union consists of multiple root classes of $\hat{f}_j$, it suffices to find a map $f$ for which, for some $j$, there are multiple classes in $\R[\varphi'_j,\psi']$ that correspond to the same class in $\R[\varphi_j,\psi]$.

\begin{ex}
Let $f$ be the $3$-valued map of the $2$-torus used as an example in \cite[Section 4]{charlotte1}, defined by the lift \[
\tilde{f}:\mathbb{R}^2\to F_3(\mathbb{R}^2,\ZZ^2):(t_1,t_2)\mapsto ((\tfrac{t_1}{2},-t_2),(\tfrac{t_1}{2}+\tfrac{1}{2},-t_2),(-t_1,-t_2+\tfrac{1}{2}))
\]
It is shown in \cite{charlotte1} that $S_1=S_2=2\ZZ\times \ZZ$ and $S_3=\ZZ^2$ (so $S=2\ZZ\times \ZZ$), and $\varphi_3:\ZZ^2\to \ZZ^2$ is given by $(z_1,z_2)\mapsto (-z_1,-z_2)$. Hence we have \begin{align*}
(z_1,z_2) \sim_{\varphi_3,\psi} (z'_1,z'_2) &\iff \exists (k,\ell) \in \ZZ^2: (z_1,z_2)=(z'_1,z'_2)+(k,\ell) \\
(z_1,z_2) \sim_{\varphi'_3,\psi'} (z'_1,z'_2) &\iff \exists (k,\ell) \in 2\ZZ\times \ZZ: (z_1,z_2)=(z'_1,z'_2)+(k,\ell).
\end{align*}
So $\R[\varphi_3,\psi]$ consists of one class, and $\R[\varphi'_3,\psi']$ consists of two classes (the pairs $(z_1,z_2)$ with $z_1$ even, and those with $z_1$ odd). In particular, for the root class $R=p\Root(\tilde{f}_3)$, the lemma implies \[
\hat{p}^{-1}(R)\cap\Root(\hat{f}_3)=\hat{R}_{[(0,0)]}\sqcup \hat{R}_{[(1,0)]}.
\]
Note that both classes $\hat{R}_{[(0,0)]}=\check{p}\Root(\tilde{f}_3)$ and $\hat{R}_{[(1,0)]}=\check{p}\Root(\tilde{f}_3+(1,0))$ are non-empty (and hence geometric): for any value of $\tilde{a}\in \RR^2$, the maps $\tilde{f}_3$ and $\tilde{f}_3+(1,0)$ have roots.
\end{ex}

Thus, the statement of \cite[Theorem 2.4(b)]{browngoncalves2023} is (in general) wrong. Consequently, some other results of \cite{browngoncalves2023} relying on this statement are no longer valid. In particular, they define the index of a root class $R$ of $f$ as the index of the root class $\hat{R}_i$ with $\hat{p}(\hat{R}_i)=R$. Since there is not necessarily a unique such class $\hat{R}_i$, this definition no longer makes sense. Also, the number $\mathrm{RR}(f)$ defined above is no longer equal to the number of algebraic root classes of $f$ (which is never claimed in \cite{browngoncalves2023} as they only consider geometric coincidence classes, but it would hold if Theorem 2.4(b) were true, and it is a natural thing to ask).

Instead, a Reidemeister theory and index for roots could be obtained by specialising our coincidence theory to the case $g=a$. The results of \cite[Section 7]{browngoncalves2023} concerning the index remain valid, since classes that can disappear after a homotopy still have index zero (cf.\ Corollary \ref{cor:index} (v)).

For the same reason, the geometric Nielsen number $\mathrm{NR}(f)$ defined above satisfies \[
N(f,a)\leq \mathrm{NR}(f)\leq \MC[f,a]=\MR[f]
\]
(where the last equality follows from Theorem \ref{thm:MC}).
In particular, by the Wecken property (Theorem \ref{thm:Wecken}):

\begin{thm}
Let $f:X\to D_n(Y)$ be an $n$-valued map between closed orientable triangulable manifolds of equal dimension $m\geq 3$. Then \[
N(f,a)=\mathrm{NR}(f)=\MR[f].
\]
\end{thm}

\end{document}